\documentclass[a4paper,11pt,oneside]{article}
\usepackage{amsfonts}
\usepackage{amsmath}
\usepackage{amssymb}
\usepackage{indentfirst}
\usepackage{bm}
\usepackage{xy}
\usepackage{epic}
\usepackage{subcaption}
\usepackage{float}
\usepackage{mathrsfs}
\usepackage{color,colortbl,amsfonts,amsmath,latexsym,bm}
\usepackage{amsthm}
\usepackage{graphicx}
\usepackage{algorithm}
\usepackage[noend]{algpseudocode}
\usepackage{enumitem}

\newtheorem{lem}{Lemma}[section]
\newtheorem{thm}{Theorem}[section]
\newtheorem{propo}{Proposition}[section]
\newtheorem{rem}{Remark}[section]

\usepackage{geometry}
\def\theequation{\thesection.\arabic{equation}}
\newcommand{\tabcaption}{\def\@captype{table}\caption}
\begin{document}

%%%%%%%%%%%%%%%%%%%%%%%%%%%%%%%%%%%%%%%%%%%%%%%%%%%%%
\title{\Large\bf  An adaptive   hybrid stress  transition quadrilateral finite element method for linear elasticity}

%\thanks{This work was supported
%Natural Science Foundation of China
%   (10771150), the National Basic Research Program of China (2005CB321701),
%    and the Program for New Century Excellent Talents in University (NCET-07-0584) }}
%\author {
% Feiteng Huang
% }

\author {
 Feiteng Huang \thanks{School of Mathematics, Sichuan University, Chengdu, 610064, China. Email: hftenger@gmail.com} \and
 Xiaoping Xie \thanks{School of Mathematics, Sichuan University, Chengdu, 610064, China. Corresponding author. Email: xpxiec@gmail.com} \and
 Chen-Song Zhang \thanks{Academy of Mathematics and System Sciences, Beijing, 100190, China. Email: zhangcs@lsec.cc.ac.cn}
%{School of Mathematics, Sichuan University, Chengdu
%610064, China}
}

\date{}
%\date{\today}
\maketitle
\begin{abstract}
%\color{red}
  In this paper, we discuss an adaptive  hybrid stress finite element method on quadrilateral meshes  for linear elasticity problems. To deal with hanging nodes arising in the adaptive  mesh refinement,  we propose new transition types of hybrid stress quadrilateral elements with 5 to 7 nodes. In particular, we derive a priori error estimation for   the 5-node transition hybrid stress element  to show that it is free from Poisson-locking, in the sense that   the error bound in the a priori estimate is independent of the Lam\'e constant $\lambda$.
   %Assumed stress hybrid methods are widely used in computational mechanics. It is well-known that they are capable to improve the performance of the standard displacement-based finite elements. When applied with adaptive mesh refinements, assumed stress hybrid methods introduce (hybrid/mixed) transition elements. In this paper, we propose a  new 5-node transition hybrid/mixed element and show that it is uniformly stable.
   We  introduce, for quadrilateral meshes, refinement/coarsening algorithms, which do not require storing the refinement tree explicitly, and    give an adaptive algorithm.  Finally  we provide some numerical results.


 \end{abstract}

\renewcommand{\theequation}{\thesection.\arabic{equation}}

%%%%%%%%%%%%%%%%%%%%%%%%%%%%%%%%%%%%%%%%%%%%%%%%%%%%%%%%%%%%%%%%%%%%%

\section{Introduction}
\setcounter{equation}{0}

Let $\Omega\subset\mathbb{R}^2$ be a convex polygonal domain, with boundary $\Gamma=\Gamma_N\cup\Gamma_D$ and
meas$(\Gamma_D)>0$.
Let $\mathbf{n}$
be the outward unit normal vector on $\Gamma$. The plane linear elasticity problem reads
\begin{equation}\label{model}
\left\{\begin{array}{ll}
-{\bf div}~\sigma = \mathbf{f} & \mbox{in} ~\Omega \\
\sigma = \mathbb{C}\varepsilon(\mathbf{u}) &\mbox{in} ~ \Omega\\
\sigma\cdot\mathbf{n}|_{\Gamma_N}=\mathbf{g}, & \mathbf{u}|_{\Gamma_D}=0
\end{array}
\right.
\end{equation}
where $\sigma\in\mathbb{R}_{\text{sym}}^{2\times2}$ is the symmetric stress tensor, $\mathbf{u}\in\mathbb{R}^2$ the
displacement field, $\varepsilon(\mathbf{u})=\frac{1}{2}(\nabla+\nabla^T)\mathbf{u}$ the strain tensor, $\mathbf{f}\in\mathbb{R}^2$ the body loading density, and $\mathbf{g}\in\mathbb{R}^2$ the surface traction.
Here $\mathbb{C}$ denotes the elasticity modulus tensor
with $\mathbb{C}\varepsilon(\mathbf{u})=2\mu\varepsilon(\mathbf{u})+\lambda\text{div}(\mathbf{u})\mathbb{I}$ and
$\mathbb{I}$ is the $2\times 2$ identity tensor. The constants $\mu, \lambda$ are the Lam$\acute{\text{e}}$
parameters, given by $\mu=\frac{E}{2(1+\nu)}, \lambda=\frac{E\nu}{(1+\nu)(1-2\nu)}$ for
plane strain problems and by $\mu=\frac{E}{2(1+\nu)}, \lambda=\frac{E\nu}{(1+\nu)(1-\nu)}$
for plane stress problems, where $0<\nu<0.5$ is the Poisson's ratio and $E$ is the Young's modulus.

Hybrid stress  finite element method (also called assumed stress hybrid finite element method), based on  Hellinger--Reissner variational principle and pioneered  by Pian~\cite{Pian1964}, is known to be an efficient  approach  ~\cite{pian1984rational,pian2000some,Pian-Wu, xie2004optimization,xie2008accurate,yu2011uniform} to  improve the performance of the standard 4-node compatible displacement quadrilateral
(bilinear) element, which yields poor results for problems with bending and, for plane strain problems, at
the nearly incompressible limit.
 In~\cite{pian1984rational} Pian and Sumihara derived a robust  4-node hybrid stress quadrilateral element (abbr. PS) through a rational choice of stress terms.
 %Pian and Tong~\cite{T.H.H. Pian P.Tong} discussed the similarity and basic difference between the incompatible  displacement model and the hybrid stress model. In the direction of determining the optimal stress parameters, there have been many other research efforts, cf.~\cite{T.H.H. Pian C.C.Wu, X.P. Xie,X.P. Xie T.X.Zhou,X.P. Xie T.X.Zhou2008,T.X.Zhou Y.F.Nie}.
 Xie and Zhou~\cite{xie2004optimization,xie2008accurate} proposed accurate 4-node hybrid stress quadrilateral elements by optimizing stress modes with a so-called energy-compatibility condition~\cite{Zhou-Nie}. Yu, Xie and Carstensen ~\cite{yu2011uniform} analyzed the methods and obtained uniform convergence and a posteriori error estimation~\cite{pian1984rational,xie2004optimization}. It  is worth noticing that the 4-node hybrid stress finite element method is of almost the same computational cost as the bilinear Q4 element due to the local elimination of stress parameters.
%To get a robust result, the assumed-stress hybrid approach is introduced. It is a kind of mixed method based on the Hellinger-Reissner variational principal
%which includes displacements and stresses. The pioneering work in this direction is by Pian\cite{pian1984rational}, where
%the assumed stress field assumed to satisfy the homogenous equilibrium equations pointwise.

Adaptive mesh refinement (AMR) for the numerical solution of the PDEs is  a standard tool
in science and engineering to achieve better accuracy with minimum degrees of
freedom.
The typical structure in one iteration of adaptive algorithms consists of four steps:
\begin{center}
\textbf{Solve $\longrightarrow$ Estimate $\longrightarrow$ Mark $\longrightarrow$ Refine/Coarsen.}
\end{center}
AMR methods locally refine/coarsen meshes according to the estimated error distribution through repeating
the above working loop comprised of finite element solution, error estimation, element (edge or patch) marking,
and mesh refinement/coarsening until the error decreases to  a prescribed level.
Classical recursive bisection and coarsening algorithms~\cite{rheinboldt1980data,samet1984quadtree,kossaczky1994recursive}
 are widely used in adaptive algorithms (see, for example, ALBERTA~\cite{schmidt2005design} and deal.II~\cite{bangerth2007deal}).
 These algorithms make use of a refinement tree data structure and subroutines to store/access the refinement history.

Chen and Zhang~\cite{Chen2010} proposed a non recursive refinement/coarsening algorithm  for triangular meshes which does not require storing the bisection
tree explicitly. They only store coordinates of vertices and connectivity of triangles which
are the minimal information required to represent a mesh for standard finite element computation. In fact, they build the bisection tree structure implicitly into a special ordering of the triangles and simplify the implementation of adaptive mesh refinement and coarsening---thus provided an
easy-access interface for the usage of mesh adaptation without much sacrifice in computing time.
These algorithms have been extended to 3D later by Bartels and Schreier~\cite{bartels2012local}.

%For adaptive quadrilateral mesh generation, there are several different approaches.
%First, the initial adaptive mesh generation method, such as Paving~\cite{blacker1991paving}.
%This kind of method requires significant foresight into the probable results of the analysis which are used to determine
%element sizes and an appropriate distribution of element density across the mesh.
Refinement and coarsening for adaptive quadrilateral meshes are more difficult than the counterparts for triangular meshes.
%As far as adaptive quadrilateral mesh refinement is concerned,
When a 4-node quadrilateral element is subdivided into four smaller elements,
hanging nodes might appear on the element boundaries of its immediate
neighborhoods.
%These neighboring elements with handing nodes are known as transition elements.
There are several different approaches to deal with the hanging nodes.
Borouchaki and Frey~\cite{borouchaki1998adaptive} presented a method to convert the
triangular mesh into a quadrilateral mesh, by which one can use
the adaptive triangular mesh generation method and then convert the mesh to a quadrilateral one.
Schneiders~\cite{schneiders1996refining} provided some template elements for local refinement
to connect   different layer patterns. This method would keep the conformity of mesh, but at the same time, could introduce distorted elements.
Another approach is to introduce transition elements, namely, keep the 'hanging' nodes in the mesh. This kind of mesh is called 1-irregular mesh,
which is widely used in the field of adaptive quadrilateral finite element methods.

Gupta~\cite{gupta1978finite}
derived a set of compatible interpolation functions for the
quadrilateral transition elements. The displacement interpolation
along a 3-node edge is continuous piecewise bilinear instead
of quadratic, thus preserves the inter-element compatibility. McDill~\cite{mcdill1987isoparametric} and Morton~\cite{morton1995new}
extended Gupta's conforming transition
elements to 3D.
Choi et al.~\cite{choi1989nonconforming,choi1993three,choi1997conforming,choi2004nonconforming}
proposed a set of 2D and 3D nonconforming transition elements.  Carstensen and Hu~\cite{carstensen2009hanging} provided a method to preserve the inter-element compatibility with just modifying the nodal bases of the immediate neighborhoods of the hanging nodes.
In~\cite{Huang2010}  Huang and Xie proved that the consistency
error of Choi and Park's 5-node nonconforming transition  quadrilateral element~\cite{choi1989nonconforming,choi1997conforming} is of only $O(h^{1/2})$-accuracy on  transition edges of the quadrilateral
subdivision. By modifying the shape functions with respect to edge mid-nodes,
the authors obtained  a  transition element with improved consistency error of order $ O(h)$.  Zhao, Shi, and Du~\cite{zhao2013constraint}
further extended the element to higher orders and
established   a posteriori error reliability and efficiency analysis.

For the plane
elasticity problem (\ref{model}), Lo, Wan, and Sze developed 4-node to 7-node  hybrid stress transition elements, using Gupta's   conforming displacement interpolation
functions~\cite{gupta1978finite}   and corresponding 5-parameter to 11-parameter stress modes in skew coordinates.
 Wu, Sze, and Lo~\cite{wu2009two} constructed, for 2D and 3D  elasticity problems, new enhanced assumed strain (EAS) and hybrid stress transition element families   with respect to the incompatible displacement modes
of Choi and Park~\cite{choi1989nonconforming,choi1997conforming}.

In this paper, basing on the incompatible displacement interpolation
functions by Huang and Xie~\cite{Huang2010}, we propose new 5-node to 7-node hybrid stress  transition quadrilateral  elements for the elasticity problem~\eqref{model} on adaptive meshes. We derive, for the presented 5-node transition element, a first-order   a priori error estimate  which is uniform with respect to the Lam\'e   constant $\lambda$.
%which can be used in the adaptive method with hybrid/mixed stress element for the linear elasticity problem~\eqref{model}.
%{\color{red} Different from Wan et al.~\cite{lo2006adaptive}, the displacement space for the transition element
% we are proposing
% is not piecewise defined on each element which still keep the optimal convergence rate. }
% Meanwhile, we present a new stress function space for 5-node transition element and prove the uniform stability.
 %namely, the element we proposed is ``locking''-free for Poisson ratio.
Besides, we introduce new refinement/coarsening algorithms  for quadrilateral meshes, which are  counterparts of the algorithms by Chen and Zhang~\cite{Chen2010} for triangular meshes. And we present an adaptive finite element method based on the proposed hybrid stress transition elements.

The rest of this paper is organized as follows. In section 2, we present   weak formulations for the plane linear
elasticity problem.  Section 3 shows the construction of new  hybrid stress  transition  elements.
Section 4  provides  new refinement/coarsening algorithms  for quadrilateral meshes and an adaptive hybrid stress finite element method.
Finally we give  some numerical experiments in Section 5.

\section{Weak formulations}
\setcounter{equation}{0}

%The modified nonconforming 5-node quadrilateral transition element can be expanded to
%6-node/7-node/8-node easily, while the consistency error is of $O(h)$, the same as 5-node transition element.

%In this section, we present some preliminaries for plane linear elasticity problem (\ref{model}).

We define the following spaces:
$$
\mathbf{V}:=\left\{\mathbf{u}\in H^1(\Omega)^2: \mathbf{u}|_{\Gamma_D}=0\right\},
$$
\begin{displaymath}
\begin{array}{l}
\Sigma:=\left\{ \begin{array}{ll} \mathbf{L}^2(\Omega;\mathbb{R}_{sym}^{2\times2}),&\text{if}~~ meas(\Gamma_N)>0,\\
\left\{\tau\in\mathbf{L}^2(\Omega;\mathbb{R}_{sym}^{2\times2}):\int_\Omega
tr\tau\mbox{d}\Omega=0\right\},&\text{if}~~ \Gamma_N=\emptyset.
\end{array}\right.
\end{array}
\end{displaymath}
Here $H^k(T)$ denotes the  usual Sobolev space consisting
of functions defined on $T$ with derivatives of order up to $k$ being square-integrable, with norm $\|\cdot\|_{k,T}$ and semi-norm $|\cdot|_{k,T}$.  In particular,  $H^0(T)=L^2(T)$.  When there
is no conflict, we may abbreviate the norm and semi-norm  to $\|\cdot\|_{k}$ and $|\cdot|_{k}$, respectively.  We use $\mathbf{L}^2(\Omega;\mathbb{R}_{sym}^{2\times2})$ to denote the space of square-integrable symmetric tensors
with the norm $\Vert\cdot\Vert_0$ defined by $\Vert\tau\Vert_0^2:=\int_\Omega\tau:\tau\text{d\bf{x}}$,
and $tr\tau:=\tau_{11}+\tau_{22}$ to represent the trace of $\tau$. We note that on the space $\mathbf{V}$ the semi-norms $|\cdot|_1$,  $|\varepsilon(\cdot)|_0$  and  the norm $\|\cdot\|_1$ are equivalent due to Korn's inequalities.
%Let ${\bf{x}}=(x,y)^T$.

Basing on the Hellinger--Reissner variational principle, the weak problem for the model (\ref{model}) reads:
Find $(\sigma,\mathbf{u})\in\Sigma\times\mathbf{V}$, such that
\begin{eqnarray}
&&a(\sigma,\tau)-b(\tau,\mathbf{u}_h)=0,~~~~\forall\tau\in\Sigma
\label{e1}
\\
&&b(\sigma,\mathbf{v})=F(\mathbf{v}),~~~~\forall\mathbf{v}\in
V \label{e2}
\end{eqnarray}
where
\begin{eqnarray}
 a(\sigma,\tau)&:=&\int_\Omega\sigma:
\mathbb{C}^{-1}\tau\mbox{d}\mathbf{x}
=\frac{1}{2\mu}\int_\Omega\left(\sigma:\tau-\frac{\lambda}{2(\mu+\lambda)}tr\sigma tr\tau\right)\mbox{d}\mathbf{x},\nonumber\\
 b(\tau,\mathbf{v})&:=&\int_\Omega\tau:\varepsilon(\mathbf{v})\mbox{d}\mathbf{x},\nonumber\\
%&=&\int_\Omega\left(\frac{1}{2\mu}\sigma^D:\tau^D+\frac{1}{4(\mu+\lambda)}tr\sigma tr\tau\right)\mbox{d}\mathbf{x},\nonumber \\
 F(\mathbf{v})&:=&\int_\Omega
\mathbf{f}\cdot\mathbf{v}\mbox{d}\mathbf{x}+\int_{\Gamma_N}\mathbf{g}\cdot\mathbf{v}\mbox{d}s.\nonumber
\end{eqnarray}
%where $\tau^D:=\tau-\frac{1}{2}tr\tau\mathbf{I}$.

%{According to the theory of mixed finite element methods~\cite{brezzi1974existence,fortin1991mixed},
As shown in~\cite{yu2011uniform}, the following two uniform stability conditions hold for the weak problem (\ref{e1}--\ref{e2}).
%{\color{blue}How about this?}
\begin{itemize}
  \item (A1) Kernel-coercivity: for any $\tau\in Z:=\{\tau\in\Sigma:\int_\Omega\tau:\mathbf{\epsilon(\mathbf{v})}\mbox{d}\mathbf{x}=0, \text{for all } \mathbf{v}\in\mathbf{V}\}$
   it holds  $\Vert\tau\Vert_0^2\lesssim a(\tau,\tau)$.
  \item (A2) Inf-sup condition: for any $\mathbf{v}\in\mathbf{V}$ it holds
   $\vert\mathbf{v}\vert_1\lesssim \sup\limits_{0\neq\tau\in\Sigma}\frac{\int_\Omega\tau:\epsilon(\mathbf{v})\mbox{d}\mathbf{x}}{\Vert\tau\Vert_0}$.
\end{itemize}
Here and in what follows, we use the notation $a\lesssim b$ (or $a\gtrsim b$)~\cite{xu1989theory}
to represent that there exists a generic positive constant $C$,
independent of the mesh parameter $h$ and  Lam$\acute{\text{e}}$ constant $\lambda$, such that $a\leq Cb$ (or $a\geq C b$). The notation $a\approx b$ abbreviates $a\lesssim b\lesssim a$.

We have the following well-posedness result; see~\cite{yu2011uniform}.
\begin{propo}\label{4convergencerate}
   Assume that $\mathbf{f}\in L^2(\Omega)^2, \mathbf{g}\in H^{1/2}(\Gamma_N)$. Then the
  weak problem (\ref{e1})--(\ref{e2}) admits a unique solution $(\sigma,\mathbf{u})\in\Sigma\cap H^1(\Omega;\mathbb{R}^{2\times2}_{sym})\times\mathbf{V}\cap H^2(\Omega)^2$ such that
  $$\vert\sigma\vert_1+\vert \mathbf{u} \vert_2\lesssim\Vert\mathbf{f}\Vert_0+\Vert\mathbf{g}\Vert_{\frac{1}{2},\Gamma_N}.$$
\end{propo}

%The weak form of this problem reads:
%
%Find $\mathbf{u}\in\mathbf{V}:=\left\{\mathbf{u}\in H^1(\Omega)^2: \mathbf{u}|_{\Gamma_D}=0\right\}$, s.t.
%\begin{equation}\label{weakform}
%a(\mathbf{u},\mathbf{v})=F(\mathbf{v}) ~~~~~~~~~~~~~~~\forall \mathbf{v}\in\mathbf{V}
%\end{equation}
%where $a(\mathbf{u},\mathbf{v})=\int_{\Omega}\left(2\mu\varepsilon(\mathbf{u}):\varepsilon(\mathbf{v})
%+\lambda\text{div}\mathbf{u}~\text{div}\mathbf{v}\right)\text{dx}$,
%$F(\mathbf{v})=\int_\Omega \mathbf{f}\cdot\mathbf{v}\text{dx}+\int_{\Gamma_N} \mathbf{g}\cdot \mathbf{v}\text{ds}$.

\section{Hybrid  stress transition  quadrilateral   elements}
\setcounter{equation}{0}
\subsection{Element geometry}

Let $\mathcal{T}_h$
be a conventional quadrilateral mesh of $\Omega$. We denote
by $h_K$ the diameter of a quadrilateral $K \in \mathcal{T}_h$, and denote
$h := \max_{K\in\mathcal{T}_h}h_K$. Let $Z_i(x_i,y_i), 1\leq i\leq4$
 be the four vertices of $K$,
and $T_i$ denotes the sub-triangle of $K$ with vertices $Z_{i-1}, Z_i$ and $Z_{i+1}$
(the index on $Z_i$ is modulo 4).

We assume that the partition $\mathcal{T}_h$ satisfies the
following ``shape-regularity'' hypothesis: there exist a constant
$\varrho>2$ independent of $h$ such that for all $K \in \mathcal{T}_h$,
\begin{equation}\label{geohyp}
  h_K\leq \varrho\rho_K
\end{equation}
with $
  \rho_K: = \min\limits_{1\leq i\leq 4} \{\text{ diameter of circle inscribed in } T_i\}.
$

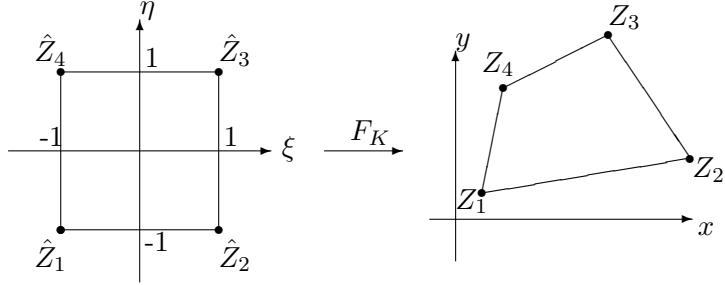
\begin{figure}[!h]
\begin{center}
\setlength{\unitlength}{0.7cm}
\begin{picture}(10,6)
\put(0,1.5){\line(1,0){3}}        \put(3,1.5){\line(0,1){3}}
\put(0,1.5){\line(0,1){3}} \put(0,4.5){\line(1,0){3}}
\put(0,1.5){\circle*{0.15}}      \put(0,4.5){\circle*{0.15}}
\put(3,4.5){\circle*{0.15}}       \put(0,1.5){\circle*{0.15}}
\put(3,1.5){\circle*{0.15}} \put(-0.5,0.8){\bf{$\hat{Z}_{1}$}}
\put(3,0.8){\bf{$\hat{Z}_{2}$}}    \put(3,4.7){\bf{$\hat{Z}_{3}$}}
\put(-0.5,4.7){\bf{$\hat{Z}_{4}$}}  \put(-1,3){\vector(1,0){5}}
\put(1.5,0.5){\vector(0,1){5}} \put(4.2,2.9){$\xi$}
\put(1.5,5.6){$\eta$}

\put(1.6,1.1){-1}                 \put(1.6,4.6){1}
\put(-0.4,3.1){-1} \put(3.1,3.1){1}

\put(5,3){\vector(1,0){1.5}}      \put(5.5,3.2){$F_K$}

\put(8,2.2){\line(6,1){4}}        \put(8,2.2){\line(1,5){0.4}}
\put(8.4,4.2){\line(2,1){2}} \put(10.4,5.2){\line(2,-3){1.6}}
\put(8,2.2){\circle*{0.15}}      \put(8.4,4.2){\circle*{0.15}}
\put(10.4,5.2){\circle*{0.15}}    \put(11.95,2.85){\circle*{0.15}}
\put(7.5,1.9){\bf{$Z_{1}$}} \put(12,2.5){\bf{$Z_{2}$}}
\put(10.4,5.4){\bf{$Z_{3}$}}       \put(8.0,4.5){\bf{$Z_{4}$}}
\put(7,1.7){\vector(1,0){5}}      \put(7.5,0.9){\vector(0,1){4}}
\put(12.1,1.4){$x$} \put(7.5,5.0){$y$}

\end{picture}
\end{center}
\vspace{-1cm} \caption{The mapping $F_{K}$}\label{rr}
\end{figure}

We define the  bilinear mapping $F_K:\widehat{K}={[-1,1]}^2\longrightarrow K$ (see Figure \ref{rr}) as
\begin{equation}\label{coordinates mapping}
{\bf x}=\left(\begin{array}{c} x\\y \end{array}\right )
=F_K(\xi,\eta)=\frac{1}{4}\sum_{i=1}^4(1+\xi_i\xi)(1+\eta_i\eta)\left(\begin{array}{c} x_i\\y_i \end{array}\right),
%\left(\begin{array}{c} \sum\limits_{i=1}^{4}x_{i}N_{i}(\xi, \eta)\\\sum\limits_{i=1}^{4}y_{i}N_{i}(\xi, \eta)) \end{array}\right),
\end{equation}
where $\xi,\eta$ are the local  coordinates, and
\begin{displaymath}
\left(\begin{array}{cccc}\xi_1 & \xi_2 & \xi_3 & \xi_4 \\ \eta_1 &
\eta_2 & \eta_3 & \eta_4 \end{array} \right )=\left (
\begin{array}{cccc} -1 & 1 & 1 & -1 \\ -1 & -1 & 1 & 1
\end{array}\right ).
\end{displaymath}
%\begin{displaymath}
%N_{1}=\frac{1}{4}(1-\xi)(1-\eta), \ \ \ N_{2}=\frac{1}{4}(1+\xi)(1-\eta),
%\end{displaymath}
%\begin{displaymath}
%N_{3}=\frac{1}{4}(1+\xi)(1+\eta), \ \ \ N_{4}=\frac{1}{4}(1-\xi)(1+\eta).
%\end{displaymath}
%We recall some geometric properties of quadrilaterals from ~\cite{yu2011uniform}.
The Jacobi matrix of the transformation $F_K$ is
\begin{displaymath}
DF_K(\xi,\eta)=\left ( \begin{array}{cc}\frac{\partial
x}{\partial\xi} & \frac{\partial x}{\partial\eta} \\ \frac{\partial
y}{\partial\xi} & \frac{\partial y}{\partial\eta}\end{array} \right
)=\left (
\begin{array}{cc} a_1+a_{12}\eta & a_2+a_{12}\xi \\ b_1+b_{12}\eta&
b_2+b_{12}\xi
\end{array}\right )
\end{displaymath}
with
\begin{displaymath}
\left( \begin{array}{cc}a_{1} & b_{1} \\ a_{2} & b_{2} \\ a_{12} & b_{12}\end{array}\right)=\frac{1}{4}\left(\begin{array}{cccc}-1 & 1 & 1 & -1 \\ -1 & -1 & 1 & 1 \\ 1 & -1 & 1 & -1 \end{array}\right)\left(\begin{array}{cc}x_{1} & y_{1} \\ x_{2} & y_{2} \\ x_{3} & y_{3} \\ x_{4} & y_{4}\end{array}\right).
\end{displaymath}
The Jacobian, $
J_K$, of $F_K$ has the form
$$
J_K(\xi,\eta)=det(DF_K)=J_0+J_1\xi+J_2\eta
$$
with
\begin{displaymath}
J_0=a_1b_2-a_2b_1,\;\;J_1=a_1b_{12}-a_{12}b_1,\;\;J_2=a_{12}b_2-a_2b_{12}.
\end{displaymath}

%\begin{rem}\rm\rm
Under the hypothesis (\ref{geohyp}), it holds the following element geometric properties (see~\cite{zhang1997analysis}): For any $K\in \mathcal{T}_h$,
  \begin{eqnarray}
\frac{\max\limits_{(\xi,\eta)\in\hat
K}J_K(\xi,\eta)}{\min\limits_{(\xi,\eta)\in\hat
K}J_K(\xi,\eta)}<\frac{h_K^2}{2\rho_K^2}\le\frac{\varrho^2}{2},\label{a0}
\end{eqnarray}
 \begin{eqnarray}
 \rho_K^2<4(a_1^2+b_1^2)<h_K^2, \label{a0}, \qquad
\rho_K^2<4(a_2^2+b_2^2)<h_K^2, \qquad %\label{a2} \\
4(a_{12}^2+b_{12}^2)<\frac{1}{4}h_K^2. \label{a3}
\end{eqnarray}

Without loss of generality, we assume
\begin{equation}
  \vert b_1\vert \leq a_1 \quad\text{and}\quad \vert a_2\vert\lesssim b_2.
\end{equation}
Then we have
\begin{eqnarray}\label{abc}
  a_1\approx b_2\approx h_K,\quad \max\{a_2,b_1\}\lesssim O(h_K),\quad
  J_K\approx J_0\approx h_K^2.
\end{eqnarray}
%\end{rem}

\subsection{ 5-node to 7-node hybrid stress  transition elements}

Let $u_i, v_i \, (i=1,...,8)$ be the two components of displacement of the four vertices and four mid-nodes
of a transition quadrilateral element $K$ (see Figure~\ref{nodenumber2d} for  nodal number systems).
Following~\cite{Huang2010}, we define the nodal basis $N_i$ ($i=1,\cdots,8$) as follows:
%\begin{eqnarray}
%  \Delta_i = \left\{
%      \begin{array}{cc} 1 & \text{if the i-th node exists(Fig.~\ref{nodenumber2d})}, \\ 0 & \text{otherwise}.
%  \end{array}
%  \right.\nonumber
%\end{eqnarray}
%
%\begin{equation}\label{N5-8}
%\begin{array}{l}
%{N}_5=\frac{3\Delta_5}{8}(1+\xi)(1-\eta^2), \ \
%{N}_6=\frac{3\Delta_6}{8}(1+\eta)(1-\xi^2), \\
%{N}_7=\frac{3\Delta_7}{8}(1-\xi)(1-\eta^2),\ \
%{N}_8=\frac{3\Delta_8}{8}(1-\eta)(1-\xi^2).
%\end{array}
%\end{equation}
\begin{equation}\label{N1-4}
\left\{\begin{array}{l}
N_1=\frac{1}{4}(1-\xi)(1-\eta)-\frac{1}{2}(\tilde{N}_7+\tilde{N}_8),\ \ N_2=\frac{1}{4}(1+\xi)(1-\eta)-\frac{1}{2}(\tilde{N}_8+\tilde{N}_5),\\
N_3=\frac{1}{4}(1+\xi)(1+\eta)-\frac{1}{2}(\tilde{N}_5+\tilde{N}_6),\ \ N_4=\frac{1}{4}(1-\xi)(1+\eta)-\frac{1}{2}(\tilde{N}_6+\tilde{N}_7),
\end{array}\right.
\end{equation}
and
 \begin{eqnarray}\label{N_i}
N_i = \Delta_i\tilde{N_i} \quad \text{ for }i=5,\cdots,8,
\end{eqnarray}
%$$N_5=\tilde{N}_i\ \ \ \ \text{ if   the $i$-th mid-side point is a node (Figure~\ref{nodenumber2d}), $i=5,\cdots,8$},$$
where
$$   \Delta_i = \left\{
      \begin{array}{ll} 1, & \text{if the $i$-th node exists (see Figure~\ref{nodenumber2d})}, \\ 0, & \text{otherwise},
  \end{array}
  \right.
$$
\begin{equation}\label{N5-8}
\left\{\begin{array}{l}
\tilde{N}_5=\frac{3}{8}(1+\xi)(1-\eta^2), \ \
\tilde{N}_6=\frac{3}{8}(1+\eta)(1-\xi^2), \\
\tilde{N}_7=\frac{3}{8}(1-\xi)(1-\eta^2),\ \
\tilde{N}_8=\frac{3}{8}(1-\eta)(1-\xi^2).
\end{array}\right.
\end{equation}

The displacement interpolation function $\mathbf{v}_{tr}$ on the transition element $K$ has the form
\begin{equation}\label{transition interpolation}
    \hat{\mathbf{v}}_{tr} = {\mathbf{v}}_{tr} \circ F_K = \sum_{i=1}^8N_i\left( \begin{array}{l}u_i \\ v_i\end{array}\right).
\end{equation}

\begin{rem}
We note that if $K$ is a normal 4-node quadrilateral element,  the displacement interpolation  $\mathbf{v}_{tr}$ reduces to the standard isoparametric bilinear interpolation $ {\mathbf{v}}_{bi}$,  i.e.
\begin{equation} \label{bilinear interpolation}
   \mathbf{v}_{tr}= \hat{\mathbf{v}}_{bi}: = {\mathbf{v}}_{bi} \circ F_K = \sum_{i=1}^4N_i\left( \begin{array}{l}u_i \\ v_i\end{array}\right).
\end{equation}
\end{rem}
%where $F_K$ is the isoparametric bilinear mapping

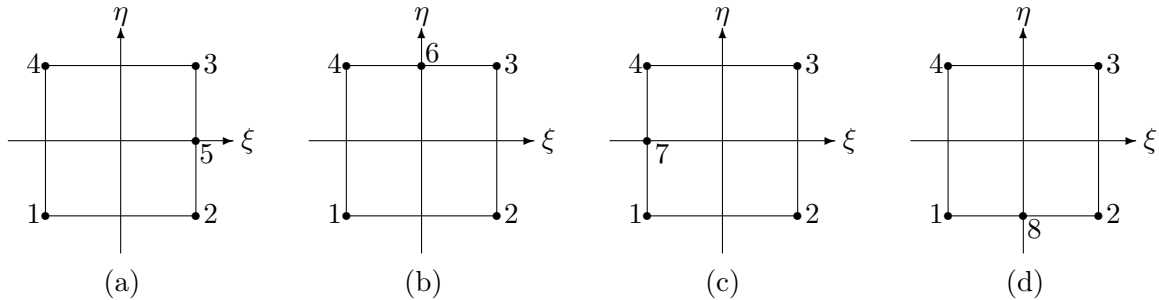
\begin{figure}[!h]
\begin{center}
\setlength{\unitlength}{1cm}
\begin{picture}(12,3)

\multiput(-1.5,1.5)(4.0,0){4}{\vector(1,0){3}}\multiput(1.6,1.40)(4.0,0){4}{$\xi$}
\multiput(0.0,0)(4.0,0){4}{\vector(0,1){3}}  \multiput(-0.1,3.1)(4.0,0){4}{$\eta$}

\multiput(-1,0.5)(4,0){4}{\line(1,0){2}} \multiput(-1,0.5)(4,0){4}{\line(0,1){2}}
\multiput(1.0,2.5)(4,0){4}{\line(-1,0){2}}\multiput(1.0,2.5)(4,0){4}{\line(0,-1){2}}
\multiput(-1,0.5)(4,0){4}{\circle*{0.10}}\multiput(1.0,0.5)(4,0){4}{\circle*{0.10}}
\multiput(1.0,2.5)(4,0){4}{\circle*{0.10}}\multiput(-1,2.5)(4,0){4}{\circle*{0.10}}

\multiput(-1.25,0.4)(4,0){4}{1}\multiput(1.1,0.4)(4,0){4}{2}
\multiput(-1.25,2.4)(4,0){4}{4}\multiput(1.1,2.4)(4,0){4}{3}
%%%%%%%%%%%%%%%%%%%%%%%%%%%%%%%%%%%%%%%%%%%%%%%%%%%%%%%%%%%%%%%
\multiput(1.0,1.5)(4,0){1}{\circle*{0.10}}
\multiput(1.05,1.20)(4,0){1}{5}

\multiput(4.0,2.5)(4,0){1}{\circle*{0.10}}
\multiput(4.05,2.55)(4,0){1}{6}

\multiput(7,1.5)(4,0){1}{\circle*{0.10}}
\multiput(7.1,1.20)(4,0){1}{7}

\put(12.0,0.5){\circle*{0.10}}
\put(12.05,0.20){8}
%%%%%%%%%%%%%%%%%%%%%%%
\put(-0.23,-0.5){(a)}
\put(3.77,-0.5){(b)}
\put(7.77,-0.5){(c)}
\put(11.77,-0.5){(d)}

\end{picture}
\end{center}
\caption{Node number system for transition elements}\label{nodenumber2d}
\end{figure}

%$$
%\left(\begin{array}{c} x\\y \end{array}\right
%)=F_K(\xi,\eta)=\frac{1}{4}\sum_{i=1}^4(1+\xi_i\xi)(1+\eta_i\eta)\left(\begin{array}{c}
%x\\y\end{array}\right)=\left(\begin{array}{c}
%a_0+a_1\xi+a_2\eta+a_{12}\xi\eta\\b_0+b_1\xi+b_2\eta+b_{12}\xi\eta\end{array}\right),
%$$and $\Phi_i=\left(\begin{array}{l}N_i\\0\end{array}\right), \Psi_i=\left(\begin{array}{l}0\\N_i\end{array}\right),$
Let $\mathbf{V_h}$ be a finite dimensional displacement space defined as
%$$\mathbf{V_h}:=\Bigg\{\mathbf{v}: \mathbf{v}|_{\Gamma_D}=0, \text{ and }\hskip8.5cm$$
\begin{equation}\label{vh} \mathbf{V_h}:=\Bigg\{\mathbf{v}:
 \mathbf{v}|_K=\mathbf{v}_{tr} \text{ for } K\in \mathcal{T}_h, \text{ and }\mathbf{v} \text{ vanishes at the nodes on }{\Gamma_D}
% \bigg\{\begin{array}{ll}
%        \mathbf{v}_{tr},  &\text{ if $K$  is a 5-node transition element in }\mathcal{T}_h\\
%        \mathbf{v}_{bi}, &\text{ if $K$  is a 4-node element in }\mathcal{T}_h
%\end{array}
\Bigg\},
\end{equation}
where   $\mathbf{v}_{tr}$  is given  by (\ref{transition interpolation}). We define, on  $\mathbf{V_h}$, a semi-norm
$$
\|\mathbf{v}\|_h:=\left(\sum\limits_{K\in\mathcal{T}_h}\int_K\nabla{\mathbf{v}}:\nabla{\mathbf{v}}d\mathbf{x}\right)^{1/2}.
$$
It is easy to see $\|\cdot\|_h$ is also a norm on  $\mathbf{V_h}$.
\begin{rem}
  From (\ref{N5-8}) and (\ref{N1-4}), it is easy to get the following relation~\cite{Huang2010}:
  \begin{equation}
    \int_e[\mathbf{w}]ds = \mathbf{0} \qquad \forall \mathbf{w} \in\mathbf{V}_h, \quad \forall e \in \mathcal{E}_h^*,\label{key}
  \end{equation}
  where   $[\mathbf{w}]$   denotes the jump of   function $ \mathbf{w}$ across an interior
edge $e$ with $[\mathbf{w}]=\mathbf{w}$ when $e\subset \partial\Omega$, and $\mathcal{E}_h^*$ is the set of all 3-node edges of all transition elements  in $\mathcal{T}_h$.

\end{rem}

%Then the corresponding finite element scheme for the problem (\ref{weakform}) reads as:
%find $\mathbf{u}_h\in\mathbf{V}_h$, such that
%\begin{equation}\label{fem}
%a_h(\mathbf{u}_h,\mathbf{v})=F(\mathbf{v}) ~~~~~~~~~~~~~~~\forall \mathbf{v}\in\mathbf{V_h}
%\end{equation}
%where $a_h(\mathbf{u}_h,\mathbf{v})=\sum\limits_{K\in\mathcal{T}_h}\int_K\left(2\mu\varepsilon(\mathbf{u}_h):\varepsilon(\mathbf{v})+\lambda\text{div}\mathbf{u}_h~\text{div}\mathbf{v}\right)\text{dx}$,
%$F(\mathbf{v})=\int_\Omega \mathbf{f}\cdot\mathbf{v}\text{dx}+\int_{\Gamma_N} \mathbf{g}\cdot \mathbf{v}\text{ds}$.

%\begin{rem}\rm
% If there are more than one mid-node at the transition element, we can define the function space the same as
% (\ref{vh}) while
% \begin{equation}
%    \hat{\mathbf{v}}_{tr} = {\mathbf{v}}_{tr} \circ F_K = \sum_{i=1}^8(\Phi_i u_i + \Psi_i v_i),\nonumber
% \end{equation}
% where the nodal basis $N_i$ $(i=5,...,8)$ is taken as follows:
% \begin{eqnarray}
%N_i = \Delta_i\tilde{N_i}, \qquad  \Delta_i = \left\{
%      \begin{array}{ll} 1, & \text{if the i-th node exists (see Fig.~\ref{nodenumber2d})}, \\ 0, & \text{otherwise}.
%  \end{array}
%  \right.
%  \nonumber
%\end{eqnarray}
%\end{rem}

In the following we introduce 5-parameter to 11-parameter stress modes corresponding to arbitrary 4-node to 7-node quadrilateral elements, with parameters $\beta_i\in \mathbb{R}$ for $i=1,2,\cdots,11$.
We use, for convenience, the Voigt notation $\tau=(\tau_{11},\tau_{22},\tau_{12})^T$ to denote a symmetric stress tensor $\tau=\left(\begin{array}{cc}\tau_{11} & \tau_{12}\\ \tau_{12} & \tau_{22}\end{array}\right)$.

\begin{description}
\item[ (1)]  If $K$ is a 4-node quadrilateral, we use the stress mode of PS~\cite{pian1984rational} or ECQ4 ~\cite{xie2004optimization} hybrid stress element with $\beta^\tau_5=(\beta_1,\ldots,\beta_5)^T$.

PS stress mode:
\begin{equation}\label{tauform4ps}
\hat{\tau}_{4}=\left( \begin{array}{ccccc}
1&0&0&\eta&\frac{a_2^2}{b^2_2}\xi\\
0&1&0&\frac{b_1^2}{a_1^2}\eta&\xi\\
0&0&1&\frac{b_1}{a_1}\eta&\frac{a_2}{b_2}\xi
\end{array}\right)\beta^\tau_5
\end{equation}

 ECQ4 stress mode:
\begin{equation}\label{tauform4ecq4}
\hat{\tau}_{4}=\left( \begin{array}{ccccc}
1-\frac{b_{12}}{b_2}\xi&\frac{a_{12}a_2}{b^2_2}\xi&\frac{a_{12}b_2-a_2b_{12}}{b^2_2}\xi&\eta&\frac{a_2^2}{b^2_2}\xi\\
\frac{b_1b_{12}}{a_1^2}\eta&1-\frac{a_{12}}{a_1}\eta&\frac{a_1b_{12}-a_{12}b_1}{a_1^2}\eta&\frac{b_1^2}{a_1^2}\eta&\xi\\
\frac{b_{12}}{a_1}\eta&\frac{a_{12}}{b_2}\xi&1-\frac{b_{12}}{b_2}\xi-\frac{a_{12}}{a_1}\eta&\frac{b_1}{a_1}\eta&\frac{a_2}{b_2}\xi
\end{array}\right)\beta^\tau_5.
\end{equation}

\item[ (2)]  If $K$ is a 5-node transition quadrilateral, we use the 7-parameter mode with $\beta^\tau_7=(\beta_1,\ldots,\beta_7)^T$:
\begin{eqnarray}\label{tauform}
\hat{\tau}_5&=&\left( \begin{array}{ccccccc}
1&0&0&\eta&0&\xi&0\\
0&1&0&0&\xi&0&\eta\\
0&0&1&\frac{b_1^2\xi+b_1b_2\eta}{a_1b_2-a_2b_1}&\frac{a_1a_2\xi+a_2^2\eta}{a_1b_2-a_2b_1}&
\frac{b_1b_2\xi+b_2^2\eta}{a_2b_1-a_1b_2}&\frac{a_1^2\xi+a_1a_2\eta}{a_2b_1-a_1b_2}
\end{array}\right)\beta^\tau_7\nonumber\\
&=:&M_7\beta_7^\tau.
\end{eqnarray}

\item[ (3)]  If $K$ is a 6-node transition quadrilateral with opposite mid-side
nodes, we use the 9-parameter mode with $\beta^\tau_9=(\beta_1,\ldots,\beta_7,\beta_8,\beta_9)^T$:
\begin{eqnarray}\label{tauform6-1}
\hat{\tau}_{6}=M_7\beta_7^\tau+
&\left( \begin{array}{cc}
2a_2^2\xi\eta-2a_1a_2\xi^2 & a_2^2\xi^2\\
2b_2^2\xi\eta-2b_1b_2\xi^2 & b_2^2\xi^2\\
2a_2b_2\xi\eta-(a_1b_2+a_2b_1)\xi^2 & a_2b_2\xi^2
\end{array}\right)
\left( \begin{array}{c}
\beta_8 \\
\beta_9
\end{array}
\right).
\end{eqnarray}

If $K$ is a 6-node transition quadrilateral with adjacent mid-side
nodes, we use the 9-parameter mode
\begin{eqnarray}\label{tauform6-2}
\hat{\tau}_{6}=M_7\beta_7^\tau+
&\left( \begin{array}{cc}
a_1^2\eta^2 & a_2^2\xi^2\\
b_1^2\eta^2 & b_2^2\xi^2\\
a_1b_1\eta^2 & a_2b_2\xi^2
\end{array}\right)
\left( \begin{array}{c}
\beta_8 \\
\beta_9
\end{array}
\right)=:M_9\beta_9^\tau.
\end{eqnarray}
\item[ (4)]  If $K$ is a 7-node transition quadrilateral, we use the 11-parameter mode
\begin{eqnarray}\label{tauform7}
\small \hat{\tau}_{7}=M_9\beta_9^\tau
+\left( \begin{array}{cc}
2a_1^2\xi\eta-2a_1a_2\eta^2 & 2a_2^2\xi\eta-2a_1a_2\xi^2\\
2b_1^2\xi\eta-2b_1b_2\eta^2 & 2b_2^2\xi\eta-2b_1b_2\xi^2\\
2a_1b_1\xi\eta-(a_1b_2+a_2b_1)\eta^2 &
2a_2b_2\xi\eta-(a_1b_2+a_2b_1)\xi^2
\end{array}\right)
\left( \begin{array}{c}
\beta_{10}\\
\beta_{11}
\end{array}
\right).
\end{eqnarray}

\end{description}

\begin{rem}
We now introduce the modified partial derivatives$\frac{\tilde\partial
\cdot}{\partial x},\frac{\tilde\partial \cdot}{\partial x}$, and corresponding $\tilde
{ {div}} \cdot,\tilde\varepsilon(\cdot)$~\cite{zhang1997analysis}: For any $K\in\mathcal{T}_h$,
\begin{eqnarray}
&(J_K\frac{\tilde\partial v}{\partial x}|_K\circ
F_K)(\xi,\eta)=\frac{\partial y}{\partial \eta}(0,0)\frac{\partial
\hat v}{\partial \xi}-\frac{\partial y}{\partial
\xi}(0,0)\frac{\partial\hat v}{\partial \eta}=b_2\frac{\partial \hat
v}{\partial \xi}-b_1\frac{\partial\hat v}{\partial \eta},&\nonumber\\
&(J_K\frac{\tilde\partial v}{\partial y}|_K\circ
F_K)(\xi,\eta)=\frac{\partial x}{\partial\xi}(0,0)\frac{\partial\hat
v}{\partial\eta}-\frac{\partial x}{\partial \eta}(0,0)\frac{\partial
\hat v}{\partial \xi}=a_1\frac{\partial\hat
v}{\partial\eta}-a_2\frac{\partial\hat v}{\partial\xi},&\nonumber\\
&\tilde{ {div}}\mathbf{v}|_K=\frac{\tilde\partial u}{\partial
x}+\frac{\tilde\partial v}{\partial y},\qquad \tilde\varepsilon(\mathbf{v})|_K=\left(\begin{array}{cc}
\frac{\tilde\partial u}{\partial x} &
\frac{1}{2}\left(\frac{\tilde\partial u}{\partial
y}+\frac{\tilde\partial v}{\partial x}\right) \\
\frac{1}{2}\left(\frac{\tilde\partial u}{\partial y} +
\frac{\tilde\partial v}{\partial x}\right) & \frac{\tilde\partial
v}{\partial y}\end{array}\right).&\nonumber
\end{eqnarray}
It is easy to know that the stress modes $\hat{\tau}_i$ defined in (\ref{tauform})--(\ref{tauform7})  satisfy the modified equilibrium relation
$$ \tilde{\bf div} \tau:=\left(\tilde{div}\left(\begin{array}{l}\tau_{11}\\\tau_{12}\end{array}\right), \tilde{div}\left(\begin{array}{l}\tau_{12}\\\tau_{22}\end{array}\right)\right)^T=0 \quad \text{ on } K $$
 for $\tau|_K=\hat{\tau}_i\circ F_K^{-1}$ and $i=5,\ldots,7$.
In particular, for the 7-parameter stress mode, $\hat{\tau}_5$ in (\ref{tauform}),  of a 5-node transition element, it's easy to verify the relation
\begin{eqnarray}\label{tauandv}
\int_K\mathbf{\tau:}\tilde\varepsilon(\mathbf{v}^b)\mbox{d}\mathbf{x}=0
\end{eqnarray}
for any $
\mathbf{v}^b \in
B_h := \left\{\mathbf{v}\in
L^2(\Omega)^2: \hat{\mathbf{v}}=\mathbf{v}^b|_K\circ F_K \in
\text{span}\{1-(\xi^2+\eta^2)/2\}^2,\forall K \in \mathcal{T}_h\right\}.
$

\end{rem}

\begin{rem}
  We note that the stress modes (\ref{tauform})--(\ref{tauform7}) for the 5-node to 7-node transition elements can be viewed as modified versions of those   introduced
  by Lo, Wan and Sze~\cite{lo2006adaptive}.  In particular,  these two versions are identical when $K$  is a parallelogram.
   \end{rem}

Basing on the stress modes (\ref{tauform4ps})--(\ref{tauform7}),  we define the approximation stress space $\Sigma_h$ as
\begin{displaymath}
\Sigma_h =\Bigg\{\tau \in
\Sigma: \hat{\tau}=\tau|_K \circ F_K =
        \hat{\tau}_i, \text{ if $K$  is a $i$-node quadrilateral  in }\mathcal{T}_h, i=4,\ldots,7\Bigg\}.
\end{displaymath}

Now we give the hybrid stress finite element scheme for the problem (\ref{e1})--(\ref{e2}):
find $(\sigma_h, \mathbf{u}_h)\in\Sigma_h\times\mathbf{V}_h$ such that
\begin{eqnarray}
&a(\sigma_h,\tau)-b_h(\tau, \mathbf{u}_h)=0, &\forall\tau\in\Sigma_h
\label{d1}
\\
&b_h(\sigma_h,\mathbf{v})=F(\mathbf{v}),& \forall\mathbf{v}\in
\mathbf{V}_h \label{d2}
\end{eqnarray}
where $b_h(\tau,\mathbf{v})=\sum\limits_{K\in\mathcal{T}_h}\int_K\tau:\varepsilon(\mathbf{v})\mbox{d}\mathbf{x}$.

\subsection{Uniform error    estimation for 5-node  hybrid stress transition element}

To derive uniform error estimates for the hybrid stress method (\ref{d1})--(\ref{d2}),
we   need, according to the mixed finite element method theory~\cite{fortin1991mixed,brezzi1974existence}, the following
two discrete versions of the uniform stability conditions (A1) and (A2):
\begin{itemize}
  \item[(A1$_h$)] Discrete Kernel-coercivity: For any $\tau\in Z_h:=\{\tau\in\Sigma_h:\sum\limits_K\int_K\tau:\mathbf{\epsilon(\mathbf{v})}\mbox{d}\mathbf{x}=0, \; \forall\mathbf{v}\in\mathbf{V}_h\}$,
   it holds that  $\Vert\tau\Vert_0^2\lesssim a(\tau,\tau)$.
  \item[(A2$_h$)] Discrete Inf-sup condition: For any $\mathbf{v}\in\mathbf{V}_h$, it holds that
   $\Vert\mathbf{v}\Vert_h\lesssim \sup\limits_{0\neq\tau\in\Sigma_h}\frac{\sum\limits_K\int_K\tau:\epsilon(\mathbf{v})\mbox{d}\mathbf{x}}{\Vert\tau\Vert_0}$.
\end{itemize}

It has been shown that the uniform stability conditions (A1$_h$)--(A2$_h$) hold in the case of 4-node hybrid stress quadrilateral finite element method~\cite{yu2011uniform}. In this subsection we will show that they also hold for the proposed  5-node  hybrid stress transition element.  For the cases of 6-node and 7-node transition elements, one may follow the same method to get similar stability results.

As for (A1$_h$),
following the same procedure as in the proof of Theorem 4.1 in~\cite{yu2011uniform} and using Theorem 5.2 of~\cite{zhang1997analysis} and (\ref{tauandv}),  we can easily obtain the following result:
\begin{propo}\label{thm1}
Let the partition $\mathcal{T}_h$ satisfy the shape-regularity
condition (\ref{geohyp}).
Assume that for any $\bar{q}\in\overline{W}_h:=\left\{\bar{q}\in L^2(\Omega): \bar{q}|_K\in P_0(K), \forall K\in\mathcal{T}_h\right\}$, there exists some $\mathbf{v}\in\mathbf{V}_h$ with
$$
\Vert \bar{q}\Vert_0^2\lesssim\int_\Omega\bar{q}div \mathbf{v}\mbox{d}\mathbf{x},\qquad  \Vert \mathbf{v}\Vert_h^2\lesssim\Vert \bar{q}\Vert_0^2.
$$
Then the uniform discrete Kernel-coercivity condition (A1$_h$) holds for the 5-node hybrid stress transition element.
\end{propo}

\begin{rem}
  The above result implies that any quadrilateral mesh which is stable for the Stokes element Q1-P0 satisfies (A1$_h$).
  As we know, the only unstable case for Q1-P0 is the checkerboard mode. Thereupon, any quadrilateral mesh which breaks
  the checkerboard mode is sufficient to guarantee the uniform stability condition (A1$_h$).
\end{rem}

The rest of this subsection is devoted to the proof of the uniform
discrete inf-sup condition (A2$_h$) for the 5-node hybrid stress transition element. Without loss of generality we only consider the cases of $(a)$ and $(c)$ in Figure~\ref{nodenumber2d}.
%In these cases,
%\begin{eqnarray}\label{5nodespace}
%\mathbf{V_h}:= \Bigg\{\text{ for all }K\in \mathcal{T}_h, \hat{\mathbf{v}}=\mathbf{v}|_K\circ F_K\in
%    \text{span}
%    \left\{1,\xi,\eta,\xi\eta,\xi\eta^2\right\}^2\Bigg\}.
%%    \left\{\left(\begin{array}{l}1\\0\end{array}\right),\left(\begin{array}{l}0\\1\end{array}\right),
%%    \left(\begin{array}{l}\xi\\0\end{array}\right),\left(\begin{array}{l}0\\\xi\end{array}\right),
%%    \left(\begin{array}{l}\eta\\0\end{array}\right),\left(\begin{array}{l}0\\\eta\end{array}\right)\right.,
%%    \quad\quad\quad\quad\quad\quad\quad\quad\quad\quad\quad\nonumber\\
%%    \left.\left(\begin{array}{l}\xi\eta\\0\end{array}\right),\left(\begin{array}{l}0\\\xi\eta\end{array}\right),
%%    \left(\begin{array}{l}\xi\eta^2\\0\end{array}\right),\left(\begin{array}{l}0\\\xi\eta^2\end{array}\right)\right\}
%%    \quad\quad\quad\quad\quad\quad\quad\quad\quad\quad\quad\quad\quad\quad\quad.\nonumber
%\end{eqnarray}
Thus,  from (\ref{N_i})--(\ref{transition interpolation}) we have, for $\mathbf{v}=(u,v)^T\in\mathbf{V_h}$ with nodal values $\mathbf{v}(Z_i)=(u_i,v_i)^T$ on $K$,
\begin{equation}\label{5nodesv}
\hat{\mathbf{v}}=\mathbf{v}\circ F_K=\sum_{i=1}^{5}
N_i\left ( \begin{array}{c}
u_i\\v_i\end{array} \right
)=:\left (
\begin{array}{c}
U_0+U_1\xi+U_2\eta+U_{12}\xi\eta+U_{122}\xi\eta^2\\
V_0+V_1\xi+V_2\eta+V_{12}\xi\eta+V_{122}\xi\eta^2
\end{array}\right ).
\end{equation}
This yields
\begin{eqnarray*}\label{JKV}
&J_K\left(\begin{array}{c}
\frac{\partial u}{\partial x}\\
\frac{\partial v}{\partial y}\\
\frac{\partial u}{\partial y}+\frac{\partial v}{\partial x}\end{array}\right)=
\left(
\begin{array}{c}
 (U_1b_2-U_2b_1)+(U_1b_{12}-U_{12}b_1)\xi+(U_{12}b_2-U_2b_{12})\eta\\
 +U_{122}(b_2\eta^2-b_{12}\xi\eta^2-2b_1\xi\eta)\\
 (V_2a_1-V_1a_2)+(V_{12}a_1-V_1a_{12})\xi+(V_2a_{12}-V_{12}a_2)\eta\\
 +V_{122}(-a_2\eta^2+a_{12}\xi\eta^2+2a_1\xi\eta)\\
 (U_2a_1-U_1a_2)+(U_{12}a_1-U_1a_{12})\xi+(U_2a_{12}-U_{12}a_2)\eta+\\
 (V_1b_2-V_2b_1)+(V_1b_{12}-V_{12}b_1)\xi+(V_{12}b_2-V_2b_{12})\eta+\\
 U_{122}(-a_2\eta^2+a_{12}\xi\eta^2+2a_1\xi\eta)+V_{122}(b_2\eta^2-b_{12}\xi\eta^2-2b_1\xi\eta)
\end{array}
\right)&
\end{eqnarray*}

\begin{eqnarray}\label{JKV}
&=\left(
\begin{array}{ccc}
b_2+b_{12}\xi & 0 & -a_2-a_{12}\xi\\
-b_1-b_{12}\eta & 0 & a_1+a_{12}\eta\\
 -b_1\xi+b_2\eta & 0 & a_1\xi-a_2\eta\\
  0 & a_1+a_{12}\eta & -b_1-b_{12}\eta\\
  0 & a_1\xi-a_2\eta & -b_1\xi+b_2\eta\\
   b_2\eta^2-b_{12}\xi\eta^2-2b_1\xi\eta & 0 & -a_2\eta^2+a_{12}\xi\eta^2+2a_1\xi\eta\\
 0  & -a_2\eta^2+a_{12}\xi\eta^2+2a_1\xi\eta & b_2\eta^2-b_{12}\xi\eta^2-2b_1\xi\eta
\end{array}
\right)^T\beta^v&\nonumber\\
\end{eqnarray}
with
$$
\beta^v=(\beta^v_1,\ldots,\beta^v_7):=
\left(U_1+\frac{b_1}{a_1}V_1, U_2+\frac{b_2}{a_1}V_1,
U_{12}+\frac{b_{12}}{a_1}V_1, V_2-\frac{a_2}{a_1}V_1,
V_{12}-\frac{a_{12}}{a_1}V_1, U_{122}, V_{122}\right)
^T.
$$

\begin{lem}\label{lem1}
  For any $\mathbf{v}\in\mathbf{V}_h$ and $K\in\mathcal{T}_h$, it holds
  \begin{equation}
  \Vert\epsilon(\mathbf{v})\Vert_{0,K}^2\lesssim \frac{1}{\min\limits_{(\xi,\eta)\in\hat{K}}J_K(\xi,\eta)}
  h_K^2\sum\limits_{1\leq i\leq 7}(\beta^v_i)^2.
  \end{equation}
\end{lem}
\begin{proof}
  From (\ref{JKV}) and (\ref{abc}), we have
  \begin{eqnarray}
  &&\Vert\epsilon(\mathbf{v})\Vert_{0,K}^2=\int_K\epsilon(\mathbf{v}):\epsilon(\mathbf{v})\mbox{d}\mathbf{x}\nonumber\\
  &&=\int_{\hat{K}}\left[
  ((b_2+b_{12}\xi)\beta^v_1-(b_1+b_{12}\eta)\beta^v_2-(b_1\xi-b_2\eta)\beta^v_3+(b_2\eta^2-b_{12}\xi\eta^2-2b_1\xi\eta)\beta^v_6)^2\right.\nonumber\\
  &&\quad+((a_1+a_{12}\eta)\beta^v_4+(a_1\xi-a_2\eta)\beta^v_5+( -a_2\eta^2+a_{12}\xi\eta^2+2a_1\xi\eta )\beta^v_7)^2\nonumber\\
  &&\quad+\frac{1}{2}(-(a_2+a_{12}\xi)\beta^v_1+(a_1+a_{12}\eta)\beta^v_2+(a_1\xi-a_2\eta)\beta^v_3-(b_1+b_{12}\eta)\beta^v_4-(b_1\xi-b_2\eta)\beta^v_5\nonumber\\
  &&\quad+\left.( -a_2\eta^2+a_{12}\xi\eta^2+2a_1\xi\eta)\beta^v_6+(b_2\eta^2-b_{12}\xi\eta^2-2b_1\xi\eta)\beta^v_7)^2
  \right]J_K^{-1}(\xi,\eta)d\xi d\eta\nonumber\\
  &&\lesssim\frac{1}{\min\limits_{(\xi,\eta)\in\hat{K}}J_K(\xi,\eta)}h_K^2\sum\limits_{1\leq i\leq 7}(\beta^v_i)^2.\nonumber
  \end{eqnarray}
\vskip -1cm
\end{proof}

\begin{lem}\label{lem2}
  For any $\tau\in\Sigma_h$ and $K\in\mathcal{T}_h$, it holds that
  \begin{equation}
    \Vert\tau\Vert_{0,K}^2\gtrsim\min\limits_{(\xi,\eta)\in\hat{K}}J_K(\xi,\eta)\sum\limits_{1\leq i\leq 7}(\beta^\tau_i)^2.
  \end{equation}
\end{lem}
\begin{proof}
  From (\ref{tauform}) and (\ref{abc}), we have
  \begin{eqnarray}
  &&\Vert\tau\Vert_{0,K}^2=\int_K\tau:\tau d\mathbf{x}\nonumber\\
  &&=\int_{\hat{K}}\left[(\beta^\tau_1+\eta\beta^\tau_4+\xi\beta^\tau_6)^2+(\beta^\tau_2+\xi\beta^\tau_5+\eta\beta^\tau_7)^2+
  2(\beta^\tau_3+
  \frac{b_1^2\xi+b_1b_2\eta}{J_0}\beta^\tau_4\right.\nonumber\\
  &&\quad+\left.
  \frac{a_1a_2\xi+a_2^2\eta}{J_0}\beta^\tau_5-
  \frac{b_1b_2\xi+b_2^2\eta}{J_0}\beta^\tau_6-\frac{a_1^2\xi+a_1a_2\eta}{J_0}\beta^\tau_7)^2
  \right]J_K(\xi,\eta)d\xi d\eta\nonumber\\
  &&\gtrsim\min\limits_{(\xi,\eta)\in\hat{K}}J_K(\xi,\eta)\sum\limits_{1\leq i\leq 7}(\beta^\tau_i)^2.\nonumber
  \end{eqnarray}
\vskip -1cm
\end{proof}

We   introduce a mesh condition given by Shi~\cite{shi1984convergence}:

\noindent\textbf{Condition (A)}. The distance $d_K \, (=2\sqrt{a_{12}^2+b_{12}^2})$ between the midpoints of the diagonals
of $K\in\mathcal{T}_h$ (see Figure \ref{conditionA}) is of order $o(h_K)$ uniformly for all elements $K$ as $h\rightarrow0$.
\begin{figure}[h!!]
  \centering
  \includegraphics[keepaspectratio,width=6cm]{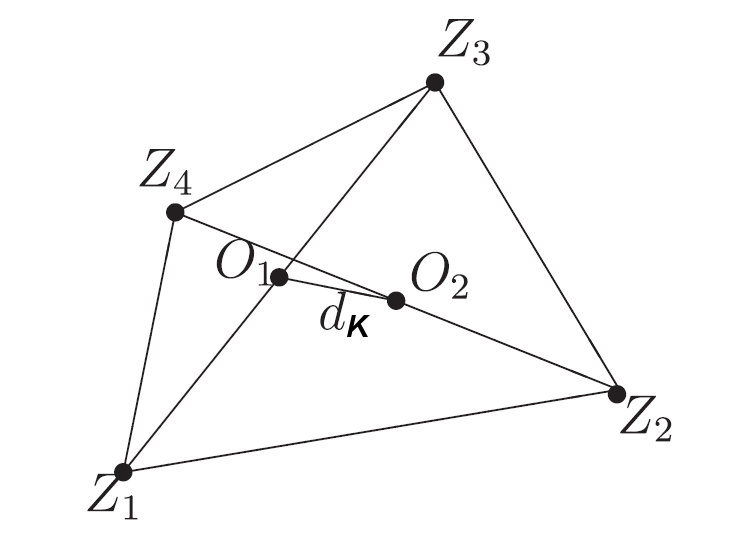}
  \vskip -12pt
  \caption{The distance $d_K$ between midpoints of two diagonals}\label{conditionA}
\end{figure}

Under this condition we have
\begin{equation}\label{conditionAhave}
\max\{\vert a_{12}\vert,\vert b_{12}\vert\}=o(h_K).
\end{equation}

\begin{lem}\label{lem3}
  Under  \textbf{Condition (A)}, for any $\mathbf{v}\in\mathbf{V}_h$ there exists a $\tau_v\in\Sigma_h$ such that for any $K\in\mathcal{T}_h$,
  \begin{equation}\label{lem3-result}
  \int_K\tau_v:\epsilon(\mathbf{v})d\mathbf{x}=\Vert\tau_v\Vert_{0,K}^2\gtrsim \Vert\epsilon(\mathbf{v})\Vert_{0,K}^2.
  \end{equation}
\end{lem}
\begin{proof}
  We follow the same line as in the proof of~\cite{yu2011uniform}. For $\tau\in\Sigma_h$ and $\mathbf{v}\in\mathbf{V}_h$, from (\ref{tauform}) and (\ref{JKV}), it holds that
  $$
  \int_K\tau:\epsilon(\mathbf{v})d\mathbf{x}=(\beta^\tau)^T A\beta^v,
  $$
  where $A=(A_1~~A_2)$ and
  \begin{eqnarray}\nonumber
  &&A_1=\left(
  \begin{array}{cccc}
   4b_2 & -4b_1 & 0 & 0 \\
   0 & 0 & 0 &4a_1  \\
   -4a_2 & 4a_1 &0 &-4b_1  \\
   -\frac{4a_{12}b_1^2}{3J_0} & \frac{4a_{12}b_1b_2-4b_{12}J_0}{3J_0} &
   \frac{4b_2J_0+4a_1b_1^2-4a_2b_2b_1}{3J_0} &
   -\frac{4b_1b_2b_{12}}{3J_0} \\
   -\frac{4a_1a_2a_{12}}{3J_0} &
   \frac{4a_2^2a_{12}}{3J_0} &
   \frac{4a_2(a_1^2-a_2^2)}{3J_0} &
   -\frac{4a_2^2b_{12}}{3J_0} \\
   \frac{4b_{12}J_0+4a_{12}b_1b_2}{3J_0} &
   -\frac{4a_{12}b_2^2}{3J_0} &
   \frac{4a_2b_2^2 - 4a_1b_1b_2 - 4b_1J_0}{3J_0} &
   \frac{4b_2^2b_{12}}{3J_0}\\
   \frac{4a_1^2a_{12}}{3J_0} &
   -\frac{4a_1a_2a_{12}}{3J_0} &
   -\frac{4a_1(a_1^2 - a_2^2)}{3J_0} &
   \frac{4a_{12}J_0+4a_1a_2b_{12}}{3J_0} \\
  \end{array}
  \right),\nonumber\\
  &&A_2=\left(
  \begin{array}{ccc}
    0 & \frac{4b_2}{3} &0\\
   0& 0 &-\frac{4a_2}{3}\\
   0 &-\frac{4a_2}{3} & \frac{4b_2}{3}\\
   -\frac{4b_1(b_1^2 - b_2^2)}{3J_0} &
   \frac{4a_{12}b_1^2}{9J_0} &-\frac{4b_1^2b_{12}}{9J_0}\\
   \frac{4a_1J_0+4b_2a_2^2-4a_1b_1a_2}{3J_0} &
   \frac{4a_1a_2a_{12}}{9J_0} &
   \frac{4a_{12}J_0-4a_1a_2b_{12}}{9J_0}\\
   \frac{4b_2(b_1^2 - b_2^2)}{3J_0} &
   - \frac{4b_{12}J_0-4a_{12}b_1b_2}{9J_0}  &
   \frac{4b_1b_2b_{12}}{9J_0}\\
   \frac{4b_1a_1^2-4a_2b_2a_1-4a_2J_0}{3J_0} &
   -\frac{4a_1^2a_{12}}{9J_0} &
   \frac{4a_1^2b_{12}}{9J_0}
  \end{array}
  \right).\nonumber
   \end{eqnarray}

   By the mean value theorem, there exists a point $(\xi_0,\eta_0)\in[-1,1]^2$ such that
   \begin{equation}
     \Vert\tau\Vert_{0,K}^2=J_k(\xi_0,\eta_0)(\beta^\tau)^TD\beta^\tau,
   \end{equation}
where
   \begin{displaymath}
       D=\left(
       \begin{array}{ccccccc}
4 & 0 & 0 &0 &0 &0& 0\\
0 & 4 & 0 &0 &0 &0 &0\\
0 &0 &8 &0 &0 & 0 &0\\
0 &0 &0 &
\frac{8b_1^4+8b_1^2b_2^2+4J_0^2}{3J_0^2}&
\frac{8a_2b_1(a_1b_1+a_2b_2)}{3J_0^2} &
-\frac{8b_1b_2(b_1^2+b_2^2)}{3J_0^2} &
-\frac{8a_1b_1(a_1b_1+a_2b_2)}{3J_0^2}\\
0 & 0 &0 &
\frac{8a_2b_1(a_1b_1+a_2b_2)}{3J_0^2} &
\frac{8a_1^2a_2^2+8a_2^4+4J_0^2}{3J_0^2}&
-\frac{8a_2b_2(a_1b_1+a_2b_2)}{3J_0^2} &
-\frac{8a_1a_2(a_1^2+a_2^2)}{3J_0^2}\\
0 &0 &0 &
-\frac{8b_1b_2(b_1^2+b_2^2)}{3J_0^2} &
-\frac{8a_2b_2(a_1b_1+a_2b_2)}{3J_0^2} &
\frac{8b_1^2b_2^2+8b_2^4+4J_0^2}{3J_0^2} &
\frac{8a_1b_2(a_1b_1+a_2b_2)}{3J_0^2}\\
0 & 0 &0 &
-\frac{8a_1b_1(a_1b_1+a_2b_2)}{3J_0^2} &
-\frac{8a_1a_2(a_1^2+a_2^2)}{3J_0^2} &
\frac{8a_1b_2(a_1b_1+a_2b_2)}{3J_0^2} &
\frac{8a_1^4+8a_1^2a_2^2+4J_0^2}{3J_0^2}
       \end{array}
       \right).
   \end{displaymath}

   By taking
   \begin{equation}\nonumber
{\tau}=\left( \begin{array}{ccccccc}
1&0&0&\eta&0&\xi&0\\
0&1&0&0&\xi&0&\eta\\
0&0&1&\frac{b_1^2\xi+b_1b_2\eta}{a_1b_2-a_2b_1}&\frac{a_1a_2\xi+a_2^2\eta}{a_1b_2-a_2b_1}&
\frac{b_1b_2\xi+b_2^2\eta}{a_2b_1-a_1b_2}&\frac{a_1^2\xi+a_1a_2\eta}{a_2b_1-a_1b_2}
\end{array}\right)\beta^{\tau,v}
\end{equation}
with
\begin{equation}\label{betatauv}
  \beta^{\tau,v}=\frac{1}{J_K(\xi,\eta)}D^{-1}A\beta^v,
\end{equation}
we immediately obtain
\begin{equation}\label{par1of46}
\int_K\tau_v:\epsilon(\mathbf{v})d\mathbf{x}=\Vert\tau_v\Vert_{0,K}^2
\end{equation}
and
$$
\beta^v=J_K(\xi_0,\eta_0)A^{-1}D\beta^{\tau,v}.
$$

From \textbf{Condition (A)} and (\ref{abc}), we see that each entry of $A$ is
$O(\frac{1}{h})$ and each entry of $D$ is $O(1)$, which implies
$$
\sum\limits_{1\leq i\leq 7}(\beta^v_i)^2\lesssim h_K^2\sum\limits_{1\leq i\leq 7}(\beta^{\tau,v}_i)^2.
$$
Combining this inequality with Lemmas~\ref{lem1}--\ref{lem2} and (\ref{abc}), we obtain
$$
\Vert\tau_v\Vert_{0,K}^2\gtrsim\Vert\epsilon(\mathbf{v})\Vert_{0,K}^2.
$$
\vskip -1cm
\end{proof}
\begin{rem}
It has been shown in~\cite{yu2011uniform} that  Lemma \ref{lem3}  holds when $K$ is a 4-node  quadrilateral, which is corresponding to the hybrid stress elements PS ~\cite{pian1984rational} or ECQ4 ~\cite{xie2004optimization}.
\end{rem}
%In view of Lemmas \ref{lem1}--\ref{lem3}, , we obtain the following theorem.
\begin{propo}\label{thm2}
  Let the partition $\mathcal{T}_h$ satisfy the shape-regularity condition (\ref{geohyp}) and
  \textbf{Condition (A)}, then the uniform discrete inf-sup condition (A2$_h$) holds for the
  5-node hybrid stress transition element.
\end{propo}
\begin{proof}
We can get the desired conclusion by following the same line as in the proof of Theorem 4.2 in~\cite{yu2011uniform}.
In fact, from Lemma \ref{lem3}, for any $\mathbf{v}\in\mathbf{V}_h$ there exists some $\tau_v\in\Sigma_h$ such that (\ref{lem3-result}) holds. This means
  \begin{eqnarray*}%\label{lem3-result}
||\tau_v||_0||\mathbf{v}||_h&\lesssim&\left( \sum\limits_K\int_K\tau_v:\tau_v d\mathbf{x}\right)^{1/2}\left( \sum\limits_K\int_K\epsilon(\mathbf{v}):\epsilon(\mathbf{v})d\mathbf{x}\right)^{1/2}
\lesssim \sum\limits_K\int_K\ \tau_v:\epsilon(\mathbf{v})d\mathbf{x}.
  \end{eqnarray*}
  Then the stability  (A2$_h$) follows immediately.
\end{proof}

Combining Propositions~\ref{thm1}--\ref{thm2} and
the standard  theory of mixed finite element methods (cf.~\cite{fortin1991mixed}),
we have the following uniform
estimate for the 5-node hybrid stress transition  element:
\begin{thm}\label{thm3}
  Let $(\sigma,\mathbf{u})\in\Sigma\times\mathbf{V}$ be the solution of the variational problem
  (\ref{e1})--(\ref{e2}). Under the conditions of Propositions~\ref{thm1}--\ref{thm2}, the discretization
  problem (\ref{d1})--(\ref{d2}) admits a unique solution $(\sigma_h,\mathbf{u}_h)\in\Sigma_h\times\mathbf{V}_h$
  such that
\begin{equation}
  \|\sigma-\sigma_h\|_0+\Vert\mathbf{u}-\mathbf{u}_h\Vert_h\lesssim\inf\limits_{\tau\in\Sigma_h}
  \|\sigma-\tau\|_0+\inf\limits_{\mathbf{v}\in
  \mathbf{V}_h}\Vert\mathbf{u}-\mathbf{v}\Vert_h+\sup\limits_{\mathbf{w}\in V_h\setminus\{0\}}\frac{
  \vert b(\sigma,\mathbf{w})-b_h(\sigma,\mathbf{w})\vert}{\|\mathbf{w}\|_h}\label{erreq}
\end{equation}
\end{thm}

For the consistency error term in the estimate (\ref{erreq}), we have
\begin{equation}
  \vert b(\sigma,\mathbf{w}) - b_h(\sigma,\mathbf{w})\vert%&=& \vert F(\mathbf{w})-b_h(\sigma,\mathbf{w})\vert\nonumber\\
  = \vert (-\text{div}\sigma, \mathbf{w})-b_h(\sigma,\mathbf{w})\vert
  = \Big\vert \sum\limits_{e\in\mathcal{E}_h^*}\int_e\sigma \mathbf{n}_e\cdot[\mathbf{w}]dx \Big\vert,\label{todo}
\end{equation}
where $\mathbf{n}_e$ is the unit outer normal vector along
 $e$.
The work left to us is to  estimate   (\ref{todo}).

Let $\mathcal{T}_h^*$ be the set of all   marco-elements, like $\tilde{K}$ in  Figure \ref{newN}, of $\mathcal{T}_h$.% each of which consists of a transition  element and its two adjacent 4-node quadrilateral elements  sharing the transition edge.
%{\color{blue}how about this?}
\begin{figure}[!h]
\begin{center}
\setlength{\unitlength}{1.5cm}
\begin{picture}(2,2)
\put(0,0.5){\line(0,1){1}} \put(0,0.5){\line(4,-1){2}}
\put(0,1.5){\line(2,1){1}} \put(1,2){\line(1,-2){1}}
\put(0,1){\line(1,0){1.5}} \put(1,0.25){\line(-1,3){0.25}}
\put(0.75,1){\circle*{0.1}}
%\put(0.375,1){\circle{0.15}}
%\put(1.125,1){\circle{0.15}}
\put(0.30,1.05){$e_1$}
\put(1.00,1.05){$e_2$}
\put(0.25,0.60){$K_1^1$}
\put(1.15,0.50){$K_1^2$} \put(0.60,1.25){$K_2$}
\put(-0.35,0.90){$\tilde{K}$}
\end{picture}
\end{center}
\vskip -10pt
\caption{micro-element $\tilde{K}$}\label{newN}
\end{figure}
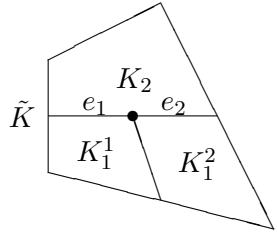
Following the same procedure as in~\cite{Huang2010}, we have the following estimate for the consistency error term.
\begin{lem}\label{lem4}
It holds that
\begin{equation}
    \sup\limits_{\mathbf{w}\in \mathbf{V_h}\setminus\{0\}}\frac{\vert b(\sigma,\mathbf{w}) - b_h(\sigma,\mathbf{w})\vert}{\|\mathbf{w}\|_h}
  \leq h|\sigma|_{1,\mathcal{T}_h^*},\label{est3Pro}
\end{equation}
where $|\sigma|_{1,\mathcal{T}_h^*}:=\Big(\sum\limits_{\tilde{K}\in\mathcal{T}_h^*}|\sigma|_{1,\tilde{K}}^2\Big)^{1/2}$.
\end{lem}
\begin{proof}
As shown in Figure \ref{newN}, we denote
\begin{eqnarray}
K_1:=K_1^1\bigcup K_1^2, \quad  \tilde{K}:=K_1\bigcup K_2,
    \end{eqnarray}
$$
 e:=e_1\bigcup e_2 \text{ with } e_i:=\overline{K_1^i}\bigcap \overline{K_2} \text{ for } i=1,2,
$$
$$  \mathbf{w}_i^e = \left(\mathbf{w}|_{K_i}\right)|_e \text{ for } i=1,2, \quad \bar{\mathbf{w}}^e = \frac{1}{|\tilde{e}|}\int_{e} \mathbf{w}_1^e\mbox{d}x.
$$
By (\ref{key}) it also holds
$
\bar{\mathbf{w}}^e=\frac{1}{|\tilde{e}|}\int_{e} \mathbf{w}_2^e\mbox{d}x.
$
 Standard scaling arguments, together with trace inequality, yields
\begin{eqnarray}\label{w-estimate}
  \vert \mathbf{w}-\bar{\mathbf{w}}^e\vert_e\lesssim h^{1/2}\vert\mathbf{w}\vert_{1,\tilde{K}}.
\end{eqnarray}

For any $\zeta\in H^1(\tilde{K})^2$,
from   trace inequality
and   Poincar$\acute{e}$ inequality it follows
\begin{eqnarray}\label{zeta-estimate}
  |\zeta - \bar{\zeta}|_{0,e} \lesssim&
  %h^{-1/2}\left(|\zeta-\mathbf{c}_1|_{0,\tilde{K}}^2+h|\zeta|_{1,\tilde{K}}|\zeta-\mathbf{c}_1|_{0,\tilde{K}}\right)^{1/2}
  h^{1/2}|\zeta|_{1,\tilde{K}},
\end{eqnarray}
where
$
\bar{\zeta}:=\frac{1}{|\tilde{K}|}\int_{\tilde{K}} \zeta\mbox{d}x.
$
The estimates (\ref{w-estimate})--(\ref{zeta-estimate}), together with (\ref{key}), imply
\begin{eqnarray}
  \Big|\int_e \zeta [\mathbf{w}]\mbox{ds}\Big|&=&\Big|\int_e(\zeta - \bar{\zeta}) [\mathbf{w} - \bar{\mathbf{w}}^e]\mbox{ds}\Big| \lesssim h|\zeta|_{1,\tilde{K}}\vert\mathbf{w}\vert_{1,\tilde{K}}.\nonumber\label{order-loss}
\end{eqnarray}
Taking $\zeta=\sigma\mathbf{n}_e$ in the above inequality and summing over all $ e\in \mathcal{E}_h^*$,  we obtain
\begin{eqnarray}
  \Big|\sum\limits_{e\in\mathcal{E}_h^*}\int_e\sigma\mathbf{n}_e[{\mathbf{w}}]\mbox{ds}\Big|\lesssim
  \sum\limits_{\tilde{K}\in\mathcal{T}_h^*}h|\sigma|_{1,\tilde{K}}\vert\mathbf{w}\vert_{1,\tilde{K}}\leq
  h|\sigma|_{1,\mathcal{T}_h^*}\|\mathbf{w}\|_{h},\nonumber
\end{eqnarray}
which yields the desired result
(\ref{est3Pro}).
\vskip -1cm
\end{proof}

From Theorem \ref{thm3}, Lemma \ref{lem4}, and the standard interpolation theory, we have the following uniform a priori error estimation.
\begin{thm}
    Let $(\sigma,\mathbf{u})\in\Sigma\bigcap H^1(\Omega;\mathbb{R}^{2\times2}_{sym})\times\mathbf{V}\bigcap H^2(\Omega)^2$ and
    $(\sigma_h,\mathbf{u}_h)\in\Sigma_h\times\mathbf{V}_h$  be respectively the solutions of the weak problem
  (\ref{e1})--(\ref{e2}) and of the discretized problem (\ref{d1})--(\ref{d2}).
  Under the same assumptions of Theorem~\ref{thm3} it holds
\begin{equation*}
    \|\mathbf{\sigma}-\mathbf{\sigma}_h\|_0+\|\mathbf{u}-\mathbf{u}_h\|_h\lesssim h(|\mathbf{\sigma}|_1
  +|\mathbf{u}|_2).
\end{equation*}
\end{thm}

%\subsection{Numerical test for the element}

\section{An adaptive algorithm for quadrilateral meshes}
\setcounter{equation}{0}

Inspired by the coarsening algorithm in~\cite{Chen2010}, we introduce new refinement/coarsening algorithms for quadrilateral meshes. Unlike the classical recursive refinement/coarsening procedures, the proposed
algorithms are non-recursive and require neither storing nor maintaining refinement
tree information such as the parents, brothers, generation, etc.
The main idea is using a special ordering of the elements in the data structure.
This also makes the implementation easier. We note that the algorithms for quadrilateral meshes are considerably more complicated than their triangular counterparts owing to the existence of hanging nodes.
%{\color{red}(FIXME: review for iFEM)}

\subsection{Data structures}

Our basic data structure for quadrilateral meshes contains five arrays, \texttt{node(1:N,1:2)}, \texttt{node$\_$flag(1:N,1)}, \texttt{edge(1:NE,1:2)}, \texttt{edge$\_$flag(1:NE,1)}, and \texttt{elem(1:NT,1:12)}, where \texttt{N} is the number of vertices,
\texttt{NE} is the number of edges, and \texttt{NT} is the number of elements.

In the node array \texttt{node}, the first and second columns contain $x-$ and $y-$coordinates of the nodes in the mesh; see Table~\ref{node}. In the \texttt{node$\_$flag}, it contains the flags for the nodes:
\texttt{`0'} for regular nodes, \texttt{`-1'} for ``newest'' nodes, \texttt{`-2'} for boundary nodes (one could define more flags such as Dirichlet Boundary, Neumann Boundary etc). A ``newest'' node refers to the internal point generated by a refinement of a quadrilateral element.
 In the edge array \texttt{edge}, the two columns contain indices to the vertices of the edge; see Table~\ref{edge}.
 In \texttt{edge$\_$flag}, the only column contains the flags for the edges:
\texttt{`-2'} for boundary edge, \texttt{`0'} for regular edge, \texttt{`$2e_f$'} if the index of this edge is bigger than its brother,
 \texttt{`$2e_f-1$'}-if index of the edge is smaller than its brother, where \texttt{`$e_c$'} is the index of ``father'' edge .
In the element array \texttt{elem}, the first four columns contain indices to the vertices of elements, the \texttt{5-8th} columns contain indices to the mid-nodes of edges of elements (\texttt{`0'}-if there is no mid-node), and the \texttt{9-12th} columns contains the edges of elements.

\begin{rem}\rm
As an
example, \texttt{node}, \texttt{node$\_$flag}, \texttt{edge}, \texttt{edge$\_$flag}, and \texttt{elem} matrices to represent a triangulation of the L-shaped domain
$\Omega=(-1, -1) \times (1, 1)\backslash([-1, 0] \times [-1, 0])$ are given in Figure~\ref{mesh4paper2do}
and Table~\ref{node}--\ref{elem}.
\end{rem}

\begin{figure}[!h]
  \centering
  \includegraphics[keepaspectratio,width=9cm]{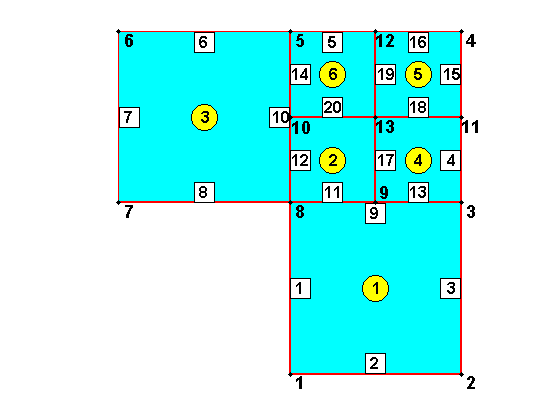}\\
  \vskip -0.7cm
  \caption{Quadrilateral mesh of L-shaped domain}\label{mesh4paper2do}
\end{figure}

\begin{table}[!h]
\begin{center}
\begin{tiny}
\begin{tabular}{c}
\texttt{x} \\
\texttt{y} \\
\texttt{flag}\\
\\
\end{tabular}
\begin{tabular}{*{13}{|c}{|}}
 \hline
0 & 1 & 1 & 1 & 0 & -1 & -1 & 0 & 0.5 & 0 & 1 & 0.5 & 0.5\\
 \hline
-1 & -1 & 0 & 1 & 1 & 1 & 0 & 0 & 0 & 0.5 & 0.5 & 1 & 0.5\\
 \hline
-2 & -2 & -2 & -2 & -2 & -2 & -2 & -2 & 0 & 0 & -2 & -2 & -1\\
 \hline
\multicolumn{1}{c}{1} &
\multicolumn{1}{c}2 &
\multicolumn{1}{c}3 &
\multicolumn{1}{c}4 &
\multicolumn{1}{c}5 &
\multicolumn{1}{c}6 &
\multicolumn{1}{c}7 &
\multicolumn{1}{c}8 &
\multicolumn{1}{c}9 &
\multicolumn{1}{c}{10} &
\multicolumn{1}{c}{11} &
\multicolumn{1}{c}{12} &
\multicolumn{1}{c}{13}\\
\end{tabular}
\end{tiny}
\end{center}
\vskip -0.7cm
\caption{\texttt{node} $\&$ \texttt{node$\_$flag} : coordinates and flag for each node}\label{node}
\end{table}

\begin{table}[!h]
\begin{tiny}
\begin{tabular}{c}
$\texttt{node}_1$ \\
$\texttt{node}_2$ \\
\texttt{flag} \\
\\
\end{tabular}
\begin{tabular}{|c|c|c|c|c|c|c|c|c|c|c|c|c|c|c|c|c|c|c|c|}
 \hline
1 & 1 & 2 & 3 & 5 & 6 & 7 & 7 & 8 & 8 & 8 & 8 & 9 & 10 & 11 & 12 & 9 & 13 & 13 & 10\\
 \hline
8 & 2 & 3 & 11 & 12 & 5 & 6 & 8 & 3 & 5 & 9 & 10 & 3 & 5 & 4 & 4 & 13 & 11 & 12 & 13\\
 \hline
 -2 & -2 & -2 & -2 & -2 & -2 & -2 & -2 & 0 & 0 & 17 & 19 & 18 & 20 & -2 & -2 & 0 & 0 & 0 & 0\\
 \hline
 \multicolumn{1}{c}1 &
 \multicolumn{1}{c}2 &
 \multicolumn{1}{c}3 &
 \multicolumn{1}{c}4 &
 \multicolumn{1}{c}5 &
 \multicolumn{1}{c}6 &
 \multicolumn{1}{c}7 &
 \multicolumn{1}{c}8 &
 \multicolumn{1}{c}9 &
 \multicolumn{1}{c}{10} &
 \multicolumn{1}{c}{11} &
 \multicolumn{1}{c}{12} &
 \multicolumn{1}{c}{13} &
 \multicolumn{1}{c}{14} &
 \multicolumn{1}{c}{15} &
 \multicolumn{1}{c}{16} &
 \multicolumn{1}{c}{17} &
 \multicolumn{1}{c}{18} &
 \multicolumn{1}{c}{19} &
 \multicolumn{1}{c}{20}
\end{tabular}
\end{tiny}
\vskip -0.4cm
\caption{\texttt{edge} $\&$ \texttt{edge$\_$flag} : edge end-point indices and edge flags}\label{edge}
\end{table}

\begin{table}[!h]
\centering
\begin{tiny}
\begin{tabular}{c}
$\texttt{elem}_1$ \\
$\texttt{elem}_2$ \\
$\texttt{elem}_3$ \\
$\texttt{elem}_4$ \\
$\texttt{elem}_5$ \\
$\texttt{elem}_6$ \\
\\
%\texttt{elem}$\_$\texttt{info}
\end{tabular}
\begin{tabular}{|c|c|c|c|c|c|c|c|c|c|c|c|}
 \hline
1 & 2 & 3 & 8 & 0 & 0 & 9 & 0 & 2 & 3 & 9 & 1\\
 \hline
8 & 9 & 13 & 10 & 0 & 0 & 0 & 0 & 11 & 17 & 20 & 12\\
 \hline
7 & 8 & 5 & 6 & 0 & 10 & 0 & 0 & 8 & 10 & 6 & 7\\
 \hline
9 & 3 & 11 & 13 & 0 & 0 & 0 & 0 & 13 & 4 & 18 & 17\\
 \hline
13 & 11 & 4 & 12 & 0 & 0 & 0 & 0 & 18 & 15 & 16 & 19\\
 \hline
10 & 13 & 12 & 5 & 0 & 0 & 0 & 0 & 20 & 19 & 5 & 14\\
 \hline
 \multicolumn{1}{c}\texttt{n}$_1$ &
 \multicolumn{1}{c}\texttt{n}$_2$ &
 \multicolumn{1}{c}\texttt{n}$_3$ &
 \multicolumn{1}{c}\texttt{n}$_4$ &
 \multicolumn{1}{c}\texttt{n}$_5$ &
 \multicolumn{1}{c}\texttt{n}$_6$ &
 \multicolumn{1}{c}\texttt{n}$_7$ &
 \multicolumn{1}{c}\texttt{n}$_8$ &
 \multicolumn{1}{c}\texttt{e}$_1$ &
 \multicolumn{1}{c}\texttt{e}$_2$ &
 \multicolumn{1}{c}\texttt{e}$_3$ &
 \multicolumn{1}{c}\texttt{e}$_4$
\end{tabular}
\end{tiny}
\vskip -0.4cm
\caption{\texttt{elem} : node and edge indices of each element}\label{elem}
\end{table}

For convenience of implementation, we also introduce two auxiliary arrays: \texttt{edge2elem} and \texttt{node2elem}. \texttt{edge2elem} (Table~\ref{edge2elem}) is a sparse matrix in IJ-format,
whose rows and columns denote the indices of elements and edges, respectively.
The $(i,j)$-entry of the matrix denotes the local index of the $j$-th \texttt{edge} in  the $i$-th \texttt{elem}.
\texttt{node2elem} (Table~\ref{node2elem}) is a sparse pattern of a sparse matrix in IJ-format, whose rows and columns denote the indices of elements and nodes respectively.

If an element contains a certain edge/node,
we say that this element is an adjacent element of the edge/node.
We now define ``good-for-coarsening'' or ``good'' node as a newest node whose adjacent elements have no hanging node. In other word, the \texttt{node$\_$flag} for a ``good'' node is `\texttt{-1}', and the \texttt{n$\_$5},...,\texttt{n$\_$8} (in \texttt{node2elem}) are all `\texttt{0}' for its adjacent elements.

\begin{table}[!h]
\centering
\begin{tiny}
\begin{tabular}{|c|c|c|c|c|c|c|c|c|c|c|c|c|c|c|c|c|c|c|c|c|c|c|c|}
 \hline
1 & 1 & 1 & 4 & 6 & 3 & 3 & 3 & 1 & 3 & 2 & 2 & 4 & 6 & 5 & 5 & 2 & 4 & 4 & 5 & 5 & 6 & 2 & 6\\
 \hline
1 & 2 & 3 & 4 & 5 & 6 & 7 & 8 & 9 & 10 & 11 & 12 & 13 & 14 & 15 & 16 & 17 & 17 & 18 & 18 & 19 & 19 & 20 & 20\\
 \hline
4 & 1 & 2 & 2 & 3 & 3 & 4 & 1 & 3 & 2 & 1 & 4 & 1 & 4 & 2 & 3 & 2 & 4 & 3 & 1 & 4 & 2 & 3 & 1\\
 \hline
% 1 & 2 & 3 & 4 & 5 & 6 & 7 & 8 & 9 & 10 & 11 & 12 & 13 & 14 & 15 & 16 & 17 & 18 & 19 & 20 & 21 & 22 & 23 & 24\\
% \hline
\end{tabular}
\end{tiny}
\vskip -0.4cm
 \caption{\texttt{edge2elem} : the first two rows denote the indices of elements and edges, respectively; the third row denotes the corresponding local indices of the edges in the element}\label{edge2elem}
\end{table}

\begin{table}[!h]
\centering
\begin{tiny}
\begin{tabular}{|c|c|c|c|c|c|c|c|c|c|c|c|c|c|c|c|c|c|c|c|c|c|c|c|}
 \hline
1 & 1 & 1 & 4 & 5 & 3 & 6 & 3 & 3 & 1 & 2 & 3 & 2 & 4 & 2 & 6 & 4 & 5 & 5 & 6 & 2 & 4 & 5 & 6\\
 \hline
1 & 2 & 3 & 3 & 4 & 5 & 5 & 6 & 7 & 8 & 8 & 8 & 9 & 9 & 10 & 10 & 11 & 11 & 12 & 12 & 13 & 13 & 13 & 13\\
% \hline
%1 & 1 & 1 & 1 & 1 & 1 & 1 & 1 & 1 & 1 & 1 & 1 & 1 & 1 & 1 & 1 & 1 & 1 & 1 & 1 & 1 & 1 & 1 & 1\\
 \hline
\end{tabular}
\end{tiny}
\vskip -0.4cm
\caption{\texttt{node2elem} : indices of elements and nodes}\label{node2elem}
\end{table}

\subsection{Refinement and coarsening algorithms}

Due to the fact that we are going to implement the algorithms in Matlab, we avoid to perform refinement/coarsening element by element.
Instead, we mark edges of all marked element, categorize these edges, and perform vectorized operations for each case. Our algorithms work as follows:
\begin{center}
\textbf{mesh + indices of marked \texttt{elem} $ \stackrel{\text{refine/coarsen}}{\xrightarrow{\hspace*{2cm}}}$ new mesh }.
\end{center}

Before refining a quadrilateral mesh, we need a post-marking step in order to make sure that there will be no more than one hanging node on each edge after the refinement.
We use \texttt{edge$\_$m} to denote the indices of marked edges,
and \texttt{elem$\_$m} to denote the indices of marked elements.
In this post-marking procedure, we first get \texttt{edge$\_$m} from \texttt{elem$\_$m}.
% namely, all edges in \texttt{elem$\_$m} belong to \texttt{edge$\_$m}.
Then we find the edges with
``hanging'' node from \texttt{edge$\_$m}, and we use \texttt{edge$\_$hg}
to denote the indices of these edges. Finally, based on \texttt{edge2elem},
we find all the elements who contain \texttt{edge$\_$hg}. By adding these elements to \texttt{elem$\_$m} we
get a new \texttt{elem$\_$m}. If the new and old \texttt{elem$\_$m} are the same, then
the post-marking procedure terminates, otherwise we do this procedure iteratively.

%{\color{red}What we do is to mark the element which contains the ``father'' edge of those elements who have already been marked to be refined.}

Now we categorize marked edges (for refinement or coarsening) into several different types. We use \texttt{elem\_adj} to denote the neighboring element(s) of marked edges, \texttt{elem2remove} to denote the elements
which will be removed by coarsening.

\noindent (1) For refinement
\begin{itemize}[leftmargin=.75in]
%\item[ Type 0:]  not marked, for all other Types the edge is marked.
\item[ Type 1:]  no hanging node and $\texttt{edge$\_$flag} = \texttt{0}$ belongs to one \texttt{edge$\_$m}.
\item[ Type 2:]  no hanging node and $\texttt{edge$\_$flag} = \texttt{0}$ belongs to two \texttt{edge$\_$m} or on the boundary.
\item[ Type 3:]  no hanging node and $\texttt{edge$\_$flag} > \texttt{0}$ is an odd number.
\item[ Type 4:]  no hanging node and $\texttt{edge$\_$flag} > \texttt{0}$ is an even number.
\item[ Type 5:]  with one hanging node.
\end{itemize}

\noindent (2) For coarsening
\begin{itemize}[leftmargin=.75in]
\item[ Type 1:] has two \texttt{elem\_adj}, and only one of them belongs to \texttt{elem2remove}.
\item[ Type 2:] has  one \texttt{elem\_adj}, but not on the boundary.
\item[ Type 3:] has two \texttt{elem\_adj}, and both of them belong to \texttt{elem2remove} or on the boundary.
\end{itemize}

%\begin{figure}[h!!]
%\begin{center}
%\setlength{\unitlength}{0.75cm}
%\begin{picture}(10,5)
%\put(0,-0.5){\line(1,0){10}}
%\put(0,4.5){\line(1,0){10}}
%\put(5,2){\line(1,0){5}}
%
%\put(0,-0.5){\line(0,1){5}}
%\put(5,-0.5){\line(0,1){5}}
%\put(7.5,-0.5){\line(0,1){5}}
%\put(10,-0.5){\line(0,1){5}}
%
%\put(-0.1,-1){1}
%\put(4.9,-1){2}
%\put(7.4,-1){3}
%\put(9.9,-1){4}
%
%\put(5.1, 1.5){5}
%\put(7.6, 1.5){6}
%\put(10.1,1.5){7}
%
%\put(-0.1, 4.6){8}
%\put(4.9,4.6){9}
%\put(7.2,4.6){10}
%\put(9.7,4.6){11}
%
%\put(0.1,1.9){e$_0$}
%\put(2.3,1.9){E$_1$}
%\put(6.05,3.1){E$_2$}
%\put(6.05,0.6){E$_3$}
%\put(8.3,3.1){E$_4$}
%\put(8.3,0.6){E$_5$}
%
%\put(4.5,1.9){e$_1$}
%\put(5.1,3.1){e$_2$}
%\put(5.1,0.6){e$_3$}
%\put(7.6,3.1){e$_4$}
%\put(7.6,0.6){e$_5$}
%\end{picture}
%\end{center}
%\caption{Edge classification}\label{fig:edge_example}
%\end{figure}

%{\color{red} No need to introduce a new example, the example in Figure 5 is good!}
The algorithm for refinement/coarsening can be found in \textbf{Algorithms \ref{refine}--\ref{coarsen}}.
To make the algorithms more accessible by readers, we use the mesh in Figure~\ref{mesh4paper2do} as an example to explain the edge types.
Let e$_i$ be the $i$--th \texttt{edge} and E$_i$ be the $i$--th \texttt{elem}. For the refinement algorithm, we can see that
\begin{itemize}
  \item If E$_6$ is marked to be refined, but not for E$_5$, then e$_{19}$ belongs to Type 1.
  \item If E$_6$ and E$_5$ are both marked to be refined, then e$_{19}$ belongs to Type 2.
  \item If E$_2$ is marked to be refined, then e$_{12}$ belongs to Type 3.
  \item If E$_6$ is marked to be refined, then e$_{14}$ belongs to Type 4.
  \item If E$_3$ is marked to be refined, then e$_7$ belongs to Type 2 and e$_{10}$ belongs to Type 5.
\end{itemize}
On the other hand, for the coarsen procedure, we have
\begin{itemize}
  \item If E$_6$ is marked for coarsening, but not for E$_5$, then e$_{19}$ belongs to Type 1.
  \item If E$_6$ is marked for coarsening, but not for E$_3$, then e$_{14}$ belongs to Type 2.
  \item If E$_6$ and E$_5$ are both marked for coarsening, then e$_{19}$ belongs to Type 3.
  \item If E$_3$ is marked for coarsening, then e$_7$ belongs to Type 3.
\end{itemize}

Now we present the algorithms for refinement and coarsening.
%{\color{red} Q: Could you put the algorithms into algorithm tex environment? It might look better. }

\begin{algorithm}[!h]
  \caption{REFINE$(mesh, \texttt{elem$\_$m})$}\label{refine}
  \begin{algorithmic}[l]
%    \Function{\textsc{Refine}$(mesh, \texttt{elem$\_$m})$}{}
            \State  1. Categorize the edges need to be refined and save as \texttt{edge}$\_$\texttt{m}.
            \State  2. Update mesh info based on the type of \texttt{edge$\_$m}:

\quad Type 1: add one node, two edges, and update adjacent element.

\quad Type 2: add one node, one edge.

\quad Type 3: add one node, two edges, and update \texttt{edge$\_$flag} information.

\quad Type 4: add one node, two edges, and update \texttt{edge$\_$flag} information.

\quad Type 5: update the elements in the patch of the edge.
%    \EndFunction
  \end{algorithmic}
\end{algorithm}

\begin{algorithm}[!h]
  \caption{COARSEN$(mesh, \texttt{elem$\_$m})$}\label{coarsen}
  \begin{algorithmic}[l]
%    \Function{\textsc{Coarsen}$(mesh, \texttt{elem$\_$m})$}{}
            \State 1. Find ``good'' nodes.
            \State 2. Mark all elements who contain the ``good'' nodes as \texttt{elem2remove}.
            \State 3. Categorize all the edges of \texttt{elem2remove}.
            \State 4. Update mesh info based on the type

\quad Type 1 : add one new edge, update \texttt{edge} info for edges belong to Type 1 $\&$ 2.

\quad Type 2 : remove one node, update \texttt{elem} info for corresponding \texttt{elem\_adj}.

\quad Type 3 : remove one node, update \texttt{edge} info for edges belong to Type 3.
%    \EndFunction
  \end{algorithmic}
\end{algorithm}

%ALGORITHM REFINE($\mathcal{T}$)
%
%\quad \quad Step 1: Classify the edges need to be refined and save as \texttt{edge}$\_$\texttt{m}
%
%
%\quad \quad Step 2: Update mesh info based on the classification of \texttt{edge$\_$m}
%
%%\quad \quad \quad  Type 0: do nothing.
%
%\quad \quad \quad  Type 1: add one node, two edges, and update adjacent element.
%
%\quad \quad \quad  Type 2: add one node, one edge.
%
%\quad \quad \quad  Type 3: add one node, two edges, and update \texttt{edge$\_$flag} information.
%
%\quad \quad \quad  Type 4: add one node, two edges, and update \texttt{edge$\_$flag} information.
%
%\quad \quad \quad  Type 5: update the elements in the patch of the edge.
%
%END

%ALGORITHM COARSEN($\mathcal{T}$)
%
%\quad \quad Step 1: Find ``good'' nodes.
%
%\quad \quad Step 2: Use ``good'' nodes to find \texttt{elem2remove}.
%
%\quad \quad Step 3: Classify all the edge of \texttt{elem2remove}.
%
%\quad \quad Step 4: Update mesh info based on the classification
%
%\quad \quad \quad  Type 1 : add one new edge, update \texttt{edge} info for edges belong to Type 1 $\&$ 2.
%
%\quad \quad \quad  Type 2 : remove one node, update \texttt{elem} info for corresponding \texttt{elem\_adj}.
%
%\quad \quad \quad  Type 3 : remove one node, update \texttt{edge} info for edges belong to Type 3.
%
%END
%{\color{blue}may need more detail...}

\subsection{Adaptive algorithm}

We are now ready to present the adaptive algorithm for discrete problem (\ref{d1})--(\ref{d2}) with
the transition hybrid stress element. The adaptive algorithm is given in \textbf{Algorithm \ref{afemA}}.

\begin{algorithm}[!h]
  \caption{AFEM}\label{afemA}
  \begin{algorithmic}[l]
%    \Function{\textsc{AFEM}}{}
            \State
            FOR $l=0,1,2,...$ UNTIL termination on level $L$, DO:
\begin{enumerate}
  \item
  Solve the discrete problem (\ref{d1})--(\ref{d2}) on $\mathcal{T}_l$;

 \item
  Compute $\eta_N:=(\sum\limits_{K\in\mathcal{T}_l}\eta_K^2)^{1/2}$ with
  $\eta_K:= \vert H_K\vert^{1/2}\vert K\vert,\forall K \in \mathcal{T}_l$ as error indicators~\cite{chen2007optimal},
  where $H_K=\frac{1}{|K|}\int_K\text{diag}(\text{abs}(\text{svd}(\nabla\sigma_{h,K})))dx+10^{-8}I$;
  %{\color{red} Q: what is $H_{\tau,p}$?}
  %\begin{center}
  %{\color{red}currently I compute the error of $\Vert\sigma-\sigma_h\Vert_0$, I may need to revise it later under Prof. Sun's help }
  %\end{center}

  %where $\mathcal{E}(K)$ is the set of edges of $K$ excluding $\partial\Omega$,
  \item
  Mark a set of elements $M_l$ in $T_l$ with minimal cardinality such that
  $
  \sum\limits_{K\in\mathcal{M}_l}\eta_K^2>\frac{1}{2}\eta_N^2;
  $
  %If $\eta_K >
%  \frac{1}{2}\max\limits_{T\in\mathcal{T}_l}\eta_T$,
%  mark $K$ for refinement;

  \item Refine $\mathcal{T}_l$ to obtain $\mathcal{T}_{l+1}$.
\end{enumerate}
%    \EndFunction
  \end{algorithmic}
\end{algorithm}

%\noindent ALGORITHM AFEM
%
%FOR $l=0,1,2,...$ UNTIL termination on level $L$, DO:
%\begin{enumerate}
%  \item
%  Solve the discrete problem (\ref{d1})--(\ref{d2}) on $\mathcal{T}_l$;
%
% \item
%  Compute $\eta_N:=(\sum\limits_{K\in\mathcal{T}_l}\eta_K^2)^{1/2}$ with
%  $
%  \eta_K=\vert \tau\vert_{H_{\tau,p}} = (\mbox{det} H_{\tau,p})^{1/2}\vert\tau\vert,
%  $ $\forall K \in \mathcal{T}_l$ as error indicators~\cite{chen2007optimal}; {\color{red} Q: what is $H_{\tau,p}$?}
%  %\begin{center}
%  %{\color{red}currently I compute the error of $\Vert\sigma-\sigma_h\Vert_0$, I may need to revise it later under Prof. Sun's help }
%  %\end{center}
%
%  %where $\mathcal{E}(K)$ is the set of edges of $K$ excluding $\partial\Omega$,
%  \item
%  Mark a set of elements $M_l$ in $T_l$ with minimal cardinality such that
%  $
%  \sum\limits_{K\in\mathcal{M}_l}\eta_K^2>\frac{1}{2}\eta_N^2;
%  $
%  %If $\eta_K >
%%  \frac{1}{2}\max\limits_{T\in\mathcal{T}_l}\eta_T$,
%%  mark $K$ for refinement;
%
%  \item Refine $\mathcal{T}_l$ to obtain $\mathcal{T}_{l+1}$.
%\end{enumerate}

\section{Numerical examples}
\setcounter{equation}{0}

In the following numerical examples, we use MATLAB (R2011a) to implement the algorithms, and the experimental platform is a desktop with Intel Xeon E5640@2.67GHz CPU and CentOS 6.5.

\subsection{Poisson's equation on an $L-$shaped domain}\label{vs1}

The domain is as in Figure~\ref{mesh4paper2do},
where $\Omega=[-1,1]^2\backslash[-1,0]^2,\
u=r^{\frac{2}{3}}\sin((2\theta+\pi)/3)$
(polar coordinate),
$\Gamma_D=\partial\Omega,\ f=0$. We call the standard h-refinement adaptive method~\cite{dorfler1996convergent}
%{\color{red}NOT a good reference for adaptivity!}
to solve the Poisson problem with the refinement algorithm proposed in the previous section for quadrilateral meshes. And we compare its performance with the bisection algorithm for triangular meshes in iFEM~\cite{Chen2008}.
Figures~\ref{aa1}--\ref{aa4} and Table~\ref{vsvs2} compare the two algorithms.
For Table~\ref{vsvs2}, we start from initial meshes of the same mesh size, and run the adaptive algorithm until the error $\vert u-u_h\vert_1< 10^{-3}$.
Here and in what following, we use DOF to denote the degree of freedom.

\begin{figure}[!h]
  \centering
  \begin{subfigure}{0.49\textwidth}
  %\centering
  \includegraphics[keepaspectratio,width=8cm]{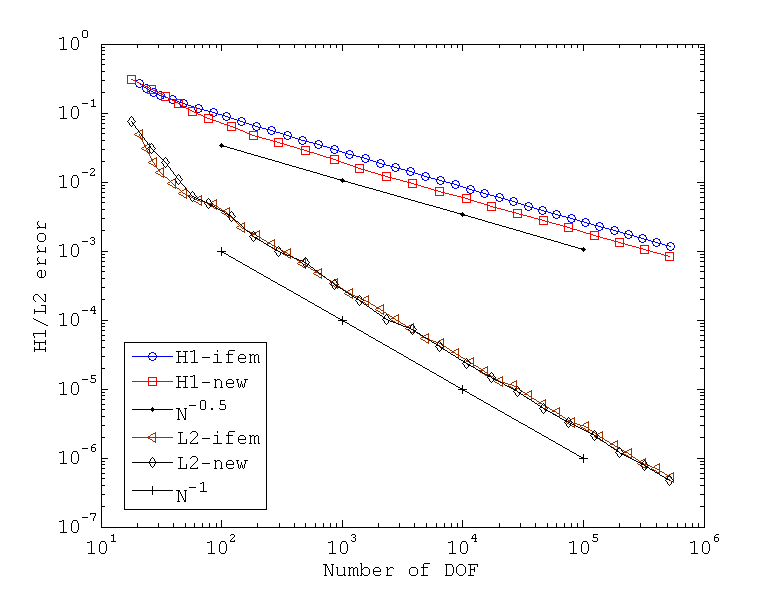}%
  %\caption{number of dof V.S. H1/L2 error}
  \end{subfigure}
 \begin{subfigure}{0.49\textwidth}
  %\centering
  \includegraphics[keepaspectratio,width=8cm]{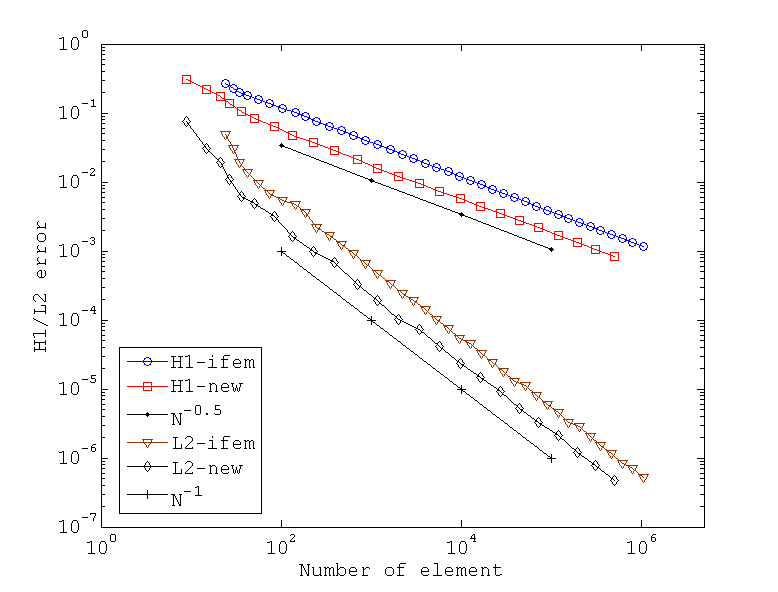}%
  %\caption{number of elem V.S. H1/L2 error}\label{aa2}
 \end{subfigure}
 \vskip -0.4cm
 \caption{Convergence rate of h-refinement for the Poisson's equation}\label{aa1}
\end{figure}

\begin{figure}[!h]
  \centering
  \begin{subfigure}{0.49\textwidth}
  \includegraphics[keepaspectratio,width=8cm]{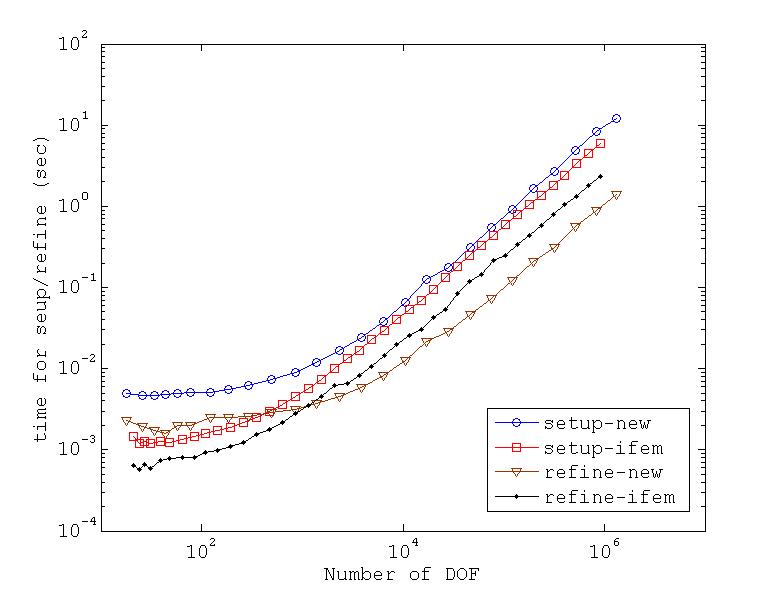}%
  %\caption{number of dof V.S. time of setup/refine}%\label{aa3}
  \end{subfigure}
  \begin{subfigure}{0.49\textwidth}
  %\centering
  \includegraphics[keepaspectratio,width=8cm]{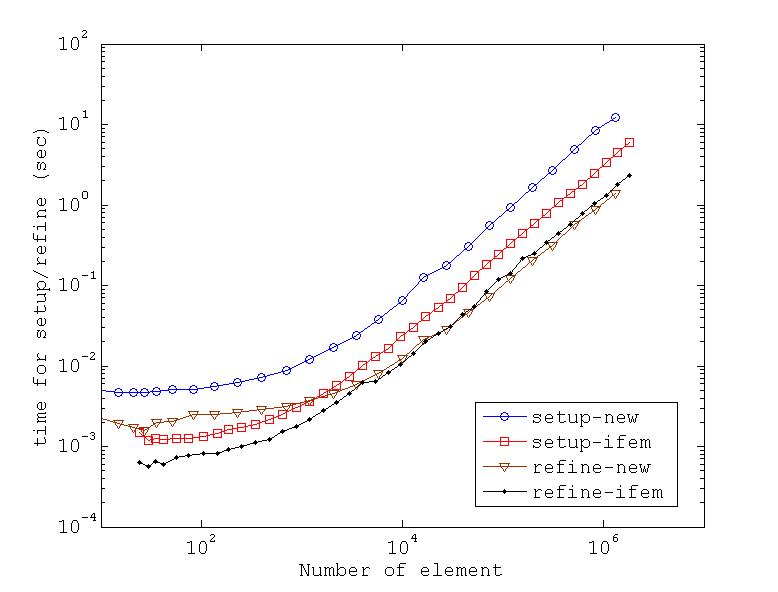}%
  %\caption{number of elem V.S. time of setup/refine}%\label{aa4}
  \end{subfigure}
  \vskip -0.4cm
  \caption{Computation time (seconds) of h-refinement for the Poisson's equation}\label{aa4}
\end{figure}

\begin{table}[!h]
  \centering
  \begin{tabular}{|c|c|c|c|}
  \hline
   & Final mesh DOF & $\vert u-u_h\vert_1$ & Total time (sec)\\
   \hline
   iFEM & 697322 & 9.97$\times 10^{-4}$ & 79.0509\\
   \hline
   NEW & 517433 & 8.31$\times 10^{-4}$ & 35.2750\\
   \hline
  \end{tabular}
  \vskip -0.4cm
  \caption{Performance of h-refinement algorithms for the Poisson's equation}\label{vsvs2}
\end{table}

\subsection{Moving circle}\label{vs2}

This example is used to test the performance of refinement and coarsening.
What we want to do is to track the interface of $x^2 + y^2 = (0.5-t)^2, t\in[0,1]$.
One of the tracking state is given in Figure~\ref{108} and the performance of the new refinement and coarsening algorithms are given in Figure~\ref{bb2}.

\begin{figure}[!h]
  \begin{minipage}[c]{0.5\textwidth}
\centering
\includegraphics[keepaspectratio,width=15cm]{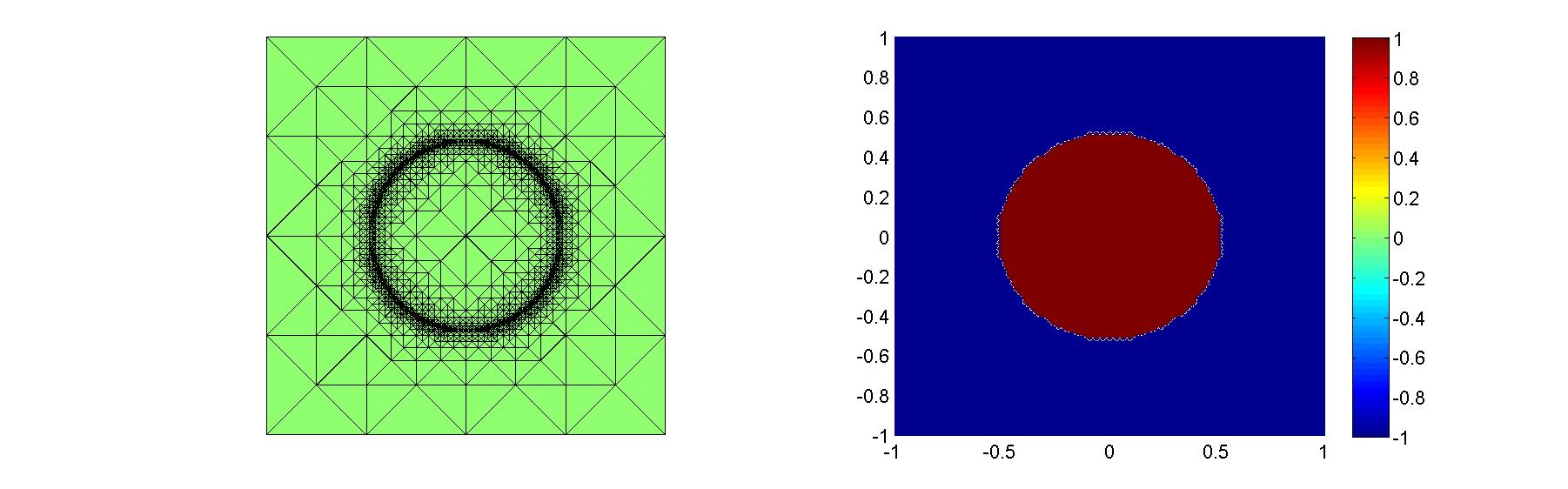}%
%\caption{ÇøÓòÆÊ·Ö}\label{L4}
\end{minipage}%

\begin{minipage}[c]{0.5\textwidth}
\centering
\includegraphics[keepaspectratio,width=15cm]{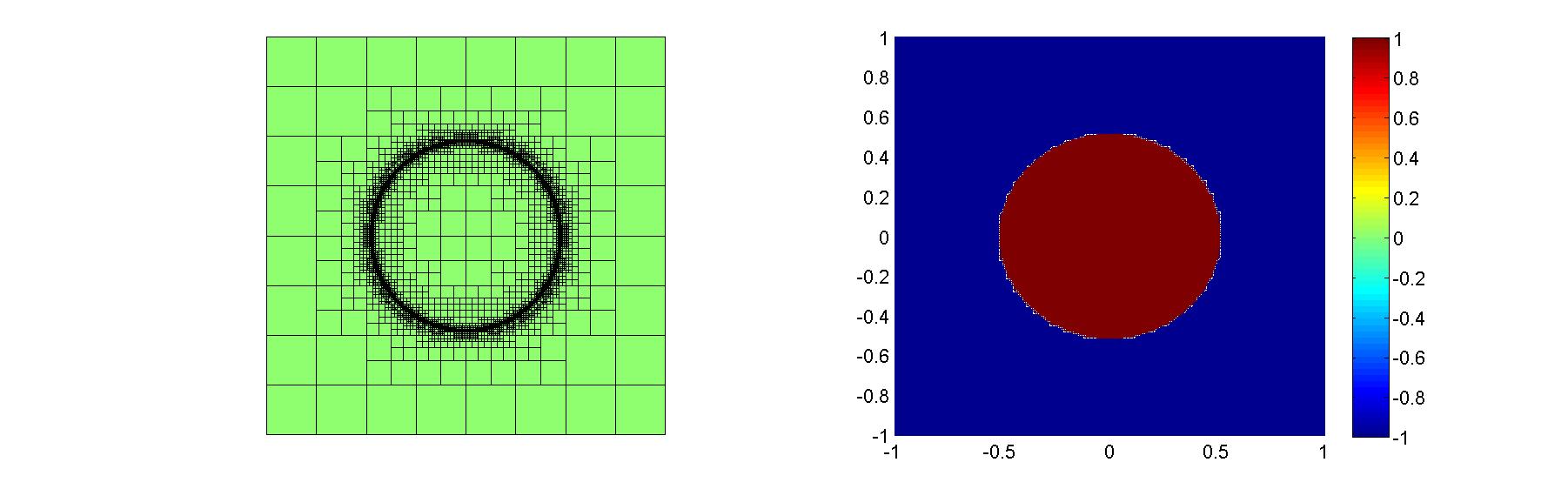}%
%\caption{node}\label{node}
\end{minipage}
\vskip -0.4cm
\caption{Tracking state for moving circle, Left: mesh; Right: interface; Upper: triangular; Lower: quadrilateral}\label{108}
\end{figure}

\begin{figure}[!h]
  \centering
  \begin{subfigure}{0.49\textwidth}
  \includegraphics[keepaspectratio,width=8cm]{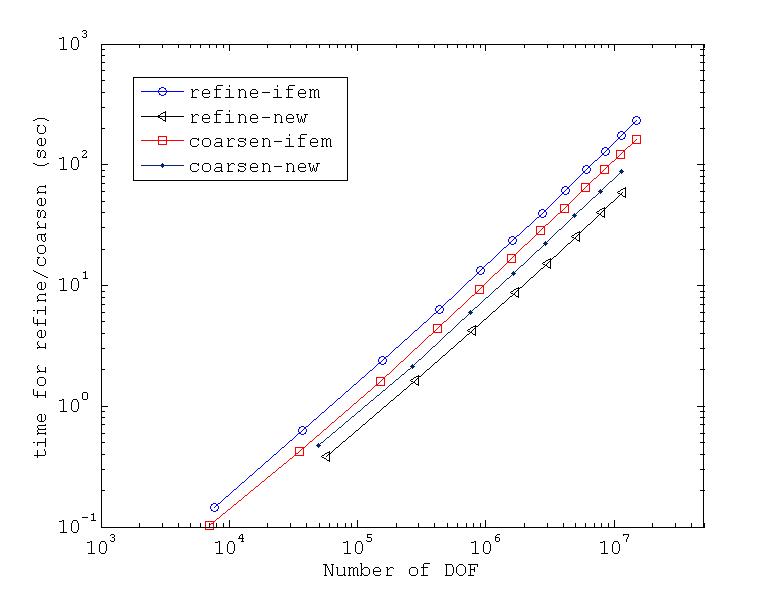}%
  %\caption{number of dof\_add/dof\_remove V.S. time of refine/coarsen}\label{bb1}
  \end{subfigure}
  \begin{subfigure}{0.49\textwidth}
  \includegraphics[keepaspectratio,width=8cm]{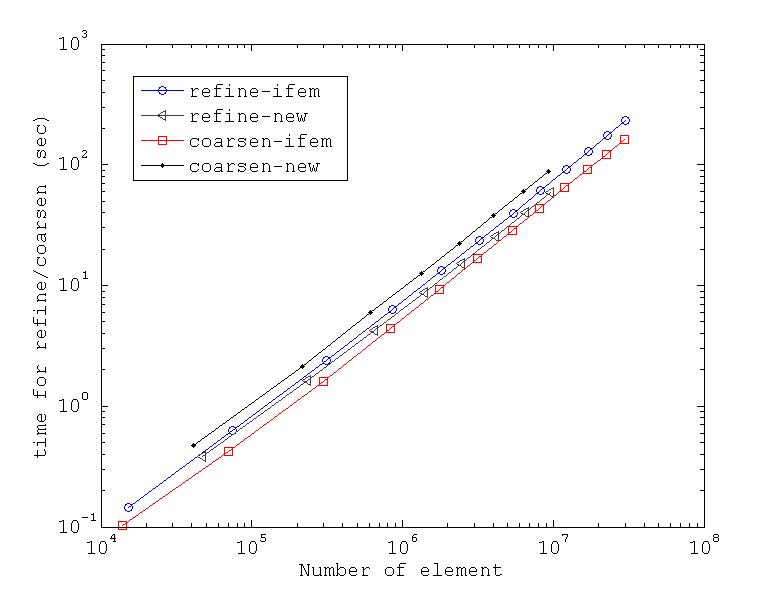}%
  %\caption{number of elem\_add/elem\_remove V.S. time of refine/coarsen}\label{bb2}
  \end{subfigure}
  \vskip -0.4cm
  \caption{Computation time (seconds) for the moving circle test}\label{bb2}
\end{figure}

\vskip5mm
The results of the above two examples in \S 5.1--5.2,  suggest that
\begin{itemize}
\item Convergence rate of the new adaptive algorithm is optimal (same as iFEM), while the errors by the new algorithm
are slightly smaller than the ones by iFEM; Furthermore, to reach same level of accuracy for solving the Poisson's equation, the adaptive quadrilateral meshes costs less computational time than the corresponding adaptive bisection meshes.
\item The new algorithms demonstrate experimentally linear computational complexity (as efficient as iFEM) in both refinement and coarsening.
%\item {\color{red}......}
\end{itemize}

\subsection{Poisson's ratio locking-free tests}\label{refine_test}
Two test problems are used to examine locking-free performance of the 5-node transition hybrid stress element.

The first one, a plane strain pure
bending cantilever beam (Figures~\ref{r}--\ref{ir}),   is a benchmark
test widely used in the literature, e.g.~\cite{pian1984rational,Pian-Wu,pian2000some, xie2004optimization,xie2008accurate,yu2011uniform,Zhou-Nie}.  The origin of
the coordinates $x, y$ is at the midpoint of the left end.
The body force $\mathbf{f}=(0,0)^T$, the surface traction $\mathbf{g}$ defined on $\Gamma_N=\{(x,y)\in[0,10]\times[-1,1]:x=10
~\text{or}~y=\pm1\}$ is given by $\mathbf{g}\mid_{x=10}=(-2Ey,0)^T,\mathbf{g}\mid_{y=\pm1}=(0,0)^T$, and
the exact solution is~\cite{yu2011uniform}
\begin{equation}\label{eqn:example1}
\mathbf{u}=\left(\begin{array}{c}-2(1-\nu^2)xy\\(1-\nu^2)x^2+\nu(1+\nu)(y^2-1)
\end{array}\right), \quad
\sigma=\left(\begin{array}{cc}-2Ey&0\\0&0
\end{array}\right).
\end{equation}

%\begin{figure}[!h]
%\begin{center}
%\setlength{\unitlength}{0.75cm}
%\begin{picture}(10,3)
%\put(0,0){\line(1,0){10}} \put(0,0){\line(0,1){2}}
%\put(0,2){\line(1,0){10}} \put(10,0){\line(0,1){2}}
%\put(1,0){\line(1,2){1}} \put(2,0){\line(1,1){2}}
%\put(4,0){\line(1,2){1}} \put(7,0){\line(-1,2){1}}
%
% \put(0,-0.2){\line(1,0){10}}
% \put(0,-0.2){\line(0,-1){0.3}}
% \put(1,-0.2){\line(0,-1){0.3}}
% \put(2,-0.2){\line(0,-1){0.3}}
% \put(4,-0.2){\line(0,-1){0.3}}
% \put(7,-0.2){\line(0,-1){0.3}}
% \put(10,-0.2){\line(0,-1){0.3}}
%
% \put(0.5,-0.7){1}
% \put(1.5,-0.7){1}
% \put(3,-0.7){2}
% \put(5.5,-0.7){3}
% \put(8.5,-0.7){3}
%
% \put(0,2.2){\line(1,0){10}}
% \put(0,2.2){\line(0,1){0.3}}
% \put(2,2.2){\line(0,1){0.3}}
% \put(4,2.2){\line(0,1){0.3}}
% \put(5,2.2){\line(0,1){0.3}}
% \put(6,2.2){\line(0,1){0.3}}
% \put(10,2.2){\line(0,1){0.3}}
%
% \put(1,2.3){2}
% \put(3,2.3){2}
% \put(4.5,2.3){1}
% \put(5.5,2.3){1}
% \put(8,2.3){4}
%
% \put(-0.3,-0.5){\line(0,1){3}}
% \multiput(-0.3,-0.5)(0,0.5){7}{\line(-1,0){0.3}}
%
% \put(-1,0){\line(0,1){2}}
% \put(-1,0){\line(-1,0){0.3}}
% \put(-1,2){\line(-1,0){0.3}}
%
% \put(-1.5,0.8){2}
%
% \put(-0.3,-0.5){\line(1,2){0.3}}
% \put(-0.3,0.5){\line(1,-2){0.3}}
%
% \put(-0.15,2){\circle{0.3}}
%
% \put(4,3){E=1500}%
%\end{picture}
%\end{center}
%\caption{Cantilever beam}
%\end{figure}

\begin{figure}[h]
\begin{center}
    \includegraphics[keepaspectratio,width=6cm]{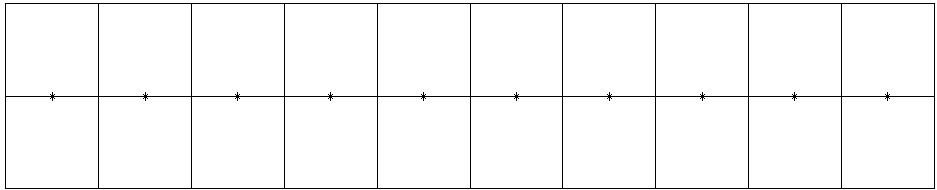}
    \includegraphics[keepaspectratio,width=6cm]{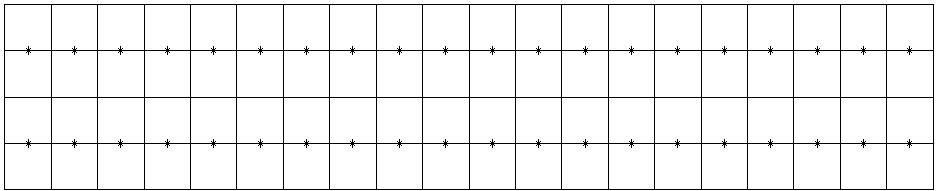}
\end{center}
\vskip -0.7cm
\caption{regular meshes} \label{r}
\end{figure}

\begin{figure}[h]
\begin{center}
    \includegraphics[keepaspectratio,width=6cm]{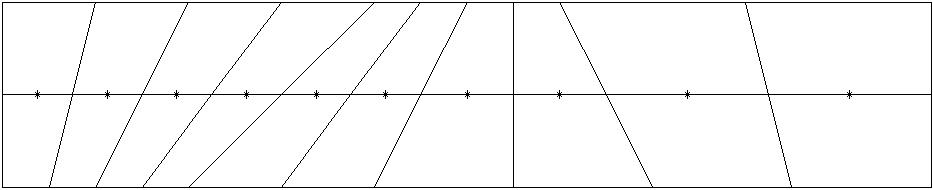}
    \includegraphics[keepaspectratio,width=6cm]{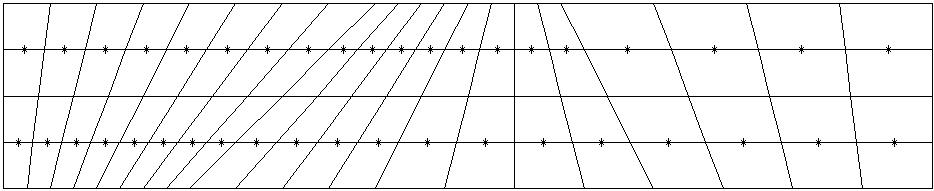}
\end{center}
\vskip -0.7cm
\caption{irregular meshes} \label{ir}
\end{figure}

The numerical results with E = 1500 and different values of
Poisson's ratio $\nu$ are listed in Tables \ref{hahatable1}--\ref{hahatable2}. The hybrid stress transition
element gives uniformly good results as Poisson's ratio $\nu\rightarrow 0.5$ or Lam\'e constant $\lambda\rightarrow \infty$, with first order accuracy for
the displacement approximation  and more than first order
accuracy for the stress approximation.
\begin{table}[!h]
\begin{center}
\begin{tabular}{*{10}{c}}
 \hline
&&regular &meshes&&&&irregular&meshes&\\
\cline{2-5}\cline{7-10}
$\nu$&$10\times2$&$20\times4$&$40\times8$&$80\times16$&&$10\times2$&$20\times4$&$40\times8$&$80\times16$\\\hline
0.49&0.0478&0.0239&0.0120&0.0060&&0.1033&0.0530&0.0268&0.0134\\
0.499&0.0496&0.0248&0.0124&0.0062&&0.1047&0.0537&0.0271&0.0136\\
0.4999&0.0497&0.0249&0.0124&0.0062&&0.1048&0.0538&0.0272&0.0136\\
0.49999&0.0497&0.0249&0.0124&0.0062&&0.1048&0.0538&0.0272&0.0136\\
0.49999999999&0.0498&0.0249&0.0124&0.0062&&0.1048&0.0538&0.0272&0.0136\\
\hline
\end{tabular}\\~\\~\\~\\
\end{center}
\vskip -50pt
\caption{$\frac{\|\mathbf{u}-\mathbf{u}_h\|_h}{|\mathbf{u}|_1}$ for {for  locking-free test 1}}\label{hahatable1}
\end{table}

\begin{table}
\begin{center}
\begin{tabular}{*{10}{c}}
\hline
&&regular&meshes&&&&irregular&meshes&\\
\cline{2-5}\cline{7-10}
$\nu$&$10\times2$&$20\times4$&$40\times8$&$80\times16$&&$10\times2$&$20\times4$&$40\times8$&$80\times16$\\\hline
0.49&1.5e-3 &7.1e-4 &3.3e-4 &9.4e-5 &&0.1018&0.0404&0.0149&0.0055\\
0.499&2.4e-4 &7.2e-5 &2.7e-5 &9.9e-6 &&0.1022&0.0419&0.0159&0.0060\\
0.4999&1.6e-5 &7.2e-6 &2.7e-6 &9.9e-7 &&0.1023&0.0421&0.0160&0.0060\\
0.49999&1.6e-6 &7.2e-7 &2.7e-7 &9.9e-8 &&0.1023&0.0421&0.0160&0.0060\\
0.49999999999&0 &0 &0 &0 &&0.1023&0.0421&0.0160&0.0060\\
\hline
\end{tabular}\\~\\~\\~\\
\end{center}
\vskip -50pt
\caption{$\frac{\|\sigma-\sigma_h\|_0}{\|\sigma\|_0}$ {for  locking-free test 1}}\label{hahatable2}
\end{table}
%\subsection{Another example for 5-node transition hybrid stress element}
%{\color{blue}I may add another example with high order polynomial exact solution to get the ``right'' convergence rate.}

In the test example above,  the stress approximation is very accurate. This is partially owing to the fact that the analytical stress solution is a linear-polynomial tensor.
We now use a more difficult plane strain test with the same domain and initial meshes as in Figures~\ref{r}--\ref{ir}. In this test, $\Gamma_N=\{(x,y)\in[0,10]\times[-1,1]:x=10
~\text{or}~y=\pm1\}$,
 $\mathbf{f}=12(\frac{x^2(1-\nu)+y^2\nu}{1-\nu^2},0)^T$,
  $\mathbf{g}\mid_{x=10}=(\frac{-4000(1-\nu)-120y^2\nu}{1-\nu^2},0)^T,\mathbf{g}\mid_{y=\pm1}=(0,0)^T$.
The exact displacement and stress solutions   are known as
\begin{equation}\label{eqn:example2}
\mathbf{u}=\frac{1}{E}\left(\begin{array}{c}
-x^4(1-\nu)-6x^2y^2\nu-\frac{y^4\nu^2}{1-\nu}
\\
4x^3y\nu+\frac{4xy^3\nu^2}{1-\nu}
\end{array}\right), \quad
\sigma=\left(\begin{array}{cc}\frac{-4x^3(1-\nu)-12xy^2\nu}{1-\nu^2}&0\\0&0
\end{array}\right).
\end{equation}
Numerical results  in Tables \ref{hahahatable1}--\ref{hahahatable2} show that the hybrid stress transition
element gives uniformly  first order accuracy for
both the displacement and stress approximations as the Poisson's ratio $\nu\rightarrow 0.5$. This is exactly what we can expect from the theory.

\begin{table}[!h]
\begin{center}
\begin{tabular}{*{10}{c}}
 \hline
&&regular &meshes&&&&irregular&meshes&\\
\cline{2-5}\cline{7-10}
$\nu$&$10\times2$&$20\times4$&$40\times8$&$80\times16$&&$10\times2$&$20\times4$&$40\times8$&$80\times16$\\\hline
0.49   & 0.1443 & 0.0720 & 0.0360 &  0.0180 && 0.1277 & 0.0632 & 0.0315 & 0.0157\\
0.499  & 0.1433 & 0.0716 & 0.0357 &  0.0179 && 0.1268 & 0.0628 & 0.0313 & 0.0156\\
0.4999 & 0.1432 & 0.0715 & 0.0357 &  0.0179 && 0.1267 & 0.0628 & 0.0313 & 0.0156\\
0.49999& 0.1432 & 0.0715 & 0.0357 &  0.0179 && 0.1267 & 0.0628 & 0.0313 & 0.0156\\
0.49999999999& 0.1432 &  0.0715 & 0.0357 & 0.0179 && 0.1267 & 0.0628 & 0.0313 &0.0156\\
\hline
\end{tabular}\\~\\~\\~\\
\end{center}
\vskip -50pt
\caption{$\frac{\|\mathbf{u}-\mathbf{u}_h\|_h}{|\mathbf{u}|_1}$  {for  locking-free test 2}}\label{hahahatable1}
\end{table}

\begin{table}
\begin{center}
\begin{tabular}{*{10}{c}}
\hline
&&regular&meshes&&&&irregular&meshes&\\
\cline{2-5}\cline{7-10}
$\nu$&$10\times2$&$20\times4$&$40\times8$&$80\times16$&&$10\times2$&$20\times4$&$40\times8$&$80\times16$\\\hline
0.49   & 0.0518 & 0.0256 & 0.0127 & 0.0064 && 0.0527 & 0.0258 & 0.0128 & 0.0064\\
0.499  & 0.0518 & 0.0256 & 0.0127 & 0.0064 && 0.0528 & 0.0258 & 0.0128 & 0.0064\\
0.4999 & 0.0518 & 0.0256 & 0.0127 & 0.0064 && 0.0528 & 0.0258 & 0.0128 & 0.0064\\
0.49999& 0.0518 & 0.0256 & 0.0127 & 0.0064 && 0.0528 & 0.0258 & 0.0128 & 0.0064\\
0.49999999999& 0.0518 & 0.0256 & 0.0127 & 0.0064 && 0.0528 & 0.0258 & 0.0128 & 0.0064\\
\hline
\end{tabular}\\~\\~\\~\\
\end{center}
\vskip -50pt
\caption{$\frac{\Vert\sigma-\sigma_h\Vert_0}{\Vert\sigma\Vert_0}$  {for  locking-free test 2}}\label{hahahatable2}
\end{table}

\subsection{Adaptive algorithm  test with transition hybrid stress elements}\label{refine_test2}
%{\color{blue}How about this?}

%For numerical tests,
We consider a square panel with edge length 2 and a one unit long edge crack~\cite{wu2009two}.
Owing to symmetry, only the upper half of the panel is analyzed; see Figure~\ref{crack_setup}.
Along the positive $x$-axis, the condition of symmetry is applied, and on other edges, traction boundary
conditions are prescribed according to the following mode I crack solution in polar coordinate~\cite{Atluri}:
$$
\sigma=
 \frac{1}{\sqrt{r}}\left\{\cos{\frac{\theta}{2}}\left(1-\sin\frac{\theta}{2}\sin\frac{3\theta}{2}\right),
  \cos{\frac{\theta}{2}}\left(1+\sin\frac{\theta}{2}\sin\frac{3\theta}{2}\right),
  \sin\frac{\theta}{2}\cos\frac{\theta}{2}\cos\frac{3\theta}{2}\right\}.
$$
The  $ \frac{1}{\sqrt{r}}$ stress singularity occurs at the crack tip.

\begin{figure}[!h]
  \centering
  \includegraphics[keepaspectratio,width=12cm]{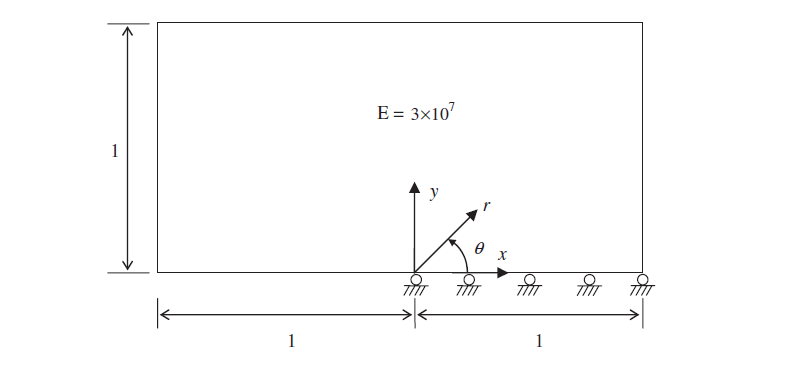}%
  \vskip -0.4cm
  \caption{Schematic diagram for half of a cracked plane strain panel. Along $x=\pm1$ and
$y=1$, exact tractions are prescribed.}\label{crack_setup}
\end{figure}
A $8\times4$ uniform mesh is taken as the initial mesh.
We show the relation between the number of DOF and the relative error $\frac{\Vert\sigma-\sigma_h\Vert}{\Vert\sigma\Vert}$ in
Figure~\ref{h_refine2}. We can see  that the stress error uniformly reduces with
a fixed factor on two successive meshes, and that the error on
the adaptively refined meshes decreases more rapidly than the one on the uniformly refined
meshes.
%The results of sections~\ref{refine_test}--\ref{refine_test2} show that the proposed 5-node transition mixed/hybrid element is
%``locking''-free for Poisson-ratio.
\begin{figure}[!h]
  \centering
  \begin{subfigure}{0.49\textwidth}
  \includegraphics[keepaspectratio,width=8cm]{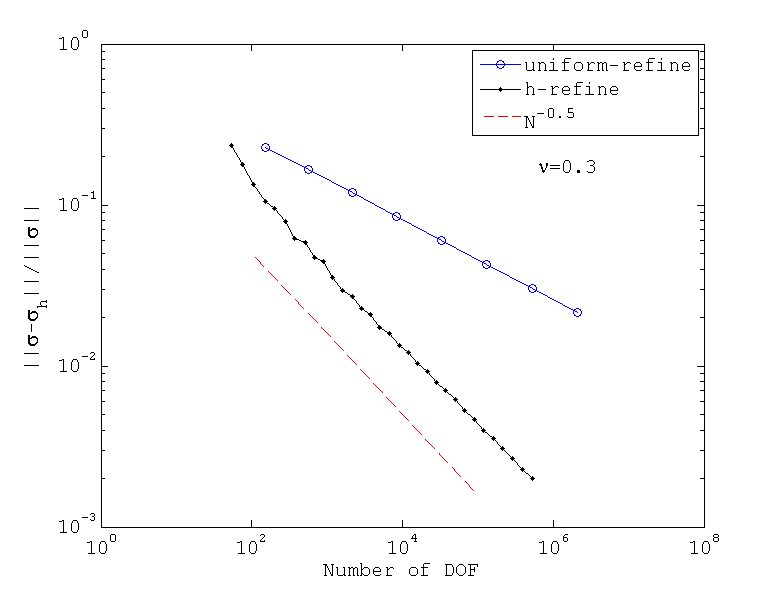}%
  %\caption{Compare h-refine method and uniform refine with $\nu=0.3$}\label{h_refine}
\end{subfigure}
\begin{subfigure}{0.49\textwidth}
  \includegraphics[keepaspectratio,width=8cm]{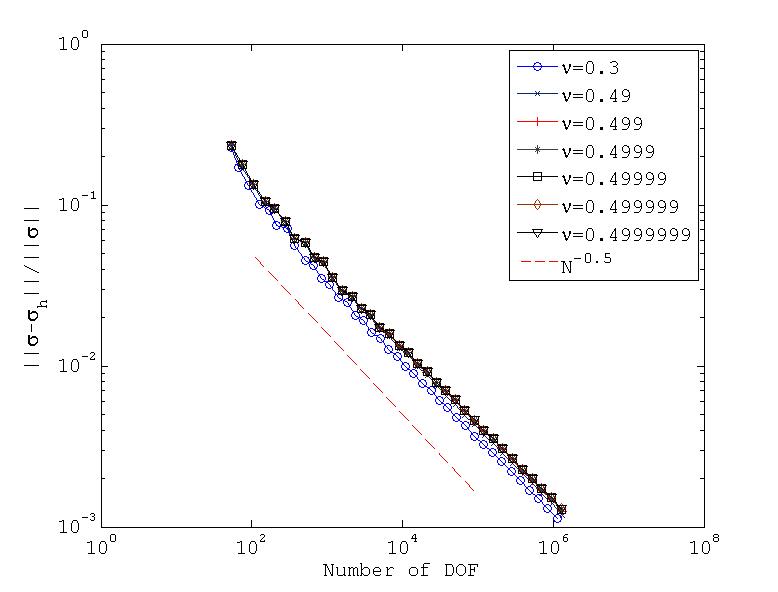}%
  %\caption{Test for different Poisson ratio}\label{h_refine2}
\end{subfigure}
\vskip -0.4cm
\caption{Convergence rates of adaptive hybrid stress transition elements  for crack problem}\label{h_refine2}
\end{figure}

%\section{Summary}
%We extend the refinement/coarsening algorithm of iFEM for quadrilateral mesh, and get the same performance. And we derive a 5-node transition hybrid/mixed element for linear elasticity problem, which is proved to be ``locking''-free. For future work, we would try to get some analysis results for the 6-node/7-node transition hybrid/mixed element.

\noindent \textbf{Acknowledgments.}
We would like to thank Prof. Pengtao Sun for his help on suggesting the error indicator used in the paper.
Huang is supported by the China Scholarship Council.  Xie is partially supported   by
the National Natural Science Foundation of China (11171239)  and
the Open Fund of  Key Laboratory of Mountain Hazards and Earth Surface Processes, CAS.
Zhang is partially supported by the NSFC Grant 91130011
and  the National High Technology Research and Development Program of China Grant 2012AA01A309.

\bibliographystyle{plain}
{\footnotesize
\bibliography{paper2do}
}
%%%%%%%%%%%%%%%%%%%%%%%%%%%%%%%%%%%%%%%%%%%%%%%%%%%%%
\end{document}